\documentclass[oneside,11pt]{amsart}
\usepackage{amsmath}
\usepackage{amssymb}
\usepackage{graphicx}
\usepackage{caption}
\usepackage{subcaption}
\usepackage{pinlabel}
\usepackage{tikz}
\usepackage[colorlinks=true]{hyperref}
\usepackage[margin=1in]{geometry}
\usepackage{soul}

\newtheorem{thm}{Theorem}[section]
\newtheorem{prop}[thm]{Proposition}
\newtheorem{lem}[thm]{Lemma}
\newtheorem{cor}[thm]{Corollary}

\newtheorem{claim}{Claim}[thm]

\theoremstyle{definition}
\newtheorem{defn}[thm]{Definition}
\newtheorem{rem}[thm]{Remark}
\newtheorem{conv}[thm]{Convention}

\newtheorem{ques}[thm]{Question}
\newtheorem{prob}[thm]{Problem}

\renewcommand{\bar}[1]{\overline{#1}}

\newcommand{\bigset}[2]{ \bigl\{ \, {#1} \bigm| {#2} \, \bigr\} }
\renewcommand{\emptyset}{\varnothing}

\renewcommand{\setminus}{-}

\newcommand{\field}[1]{\mathbb{#1}}
\newcommand{\Z}{\field{Z}}

\newcommand{\R}{\field{R}}

\renewcommand{\implies}{\Rightarrow}

\usepackage{ifthen}

\newcommand{\showcomments}{yes}

\newsavebox{\commentbox}
{\ifthenelse{\equal{\showcomments}{yes}}%
{\footnotemark
    \begin{lrbox}{\commentbox}
    \begin{minipage}[t]{1.25in}\raggedright\sffamily\upshape\tiny
    \footnotemark[\arabic{footnote}]}%
{\begin{lrbox}{\commentbox}}}%
{\ifthenelse{\equal{\showcomments}{yes}}%
{\end{minipage}\end{lrbox}\marginpar{\usebox{\commentbox}}}%
{\end{lrbox}}}

\title[\tiny Croke-Kleiner admissible groups: Property (QT) and quasiconvexity] {\footnotesize Croke-Kleiner admissible groups: Property (QT) and quasiconvexity}

\author{Hoang Thanh Nguyen}
\address{Beijing International Center for Mathematical Research\\
Peking University\\
 Beijing 100871, China
P.R.}
\email{nthoang.math@gmail.com}

\author{Wenyuan Yang}
\address{Beijing International Center for Mathematical Research\\
Peking University\\
 Beijing 100871, China
P.R.}
\email{wyang@bicmr.pku.edu.cn}

\date{\today}
\begin{document}

\begin{abstract}
Croke-Kleiner admissible groups firstly introduced by Croke-Kleiner in \cite{CK02} belong to a particular class of graph of groups which generalize fundamental groups of $3$--dimensional graph manifolds.
In this paper, we show that if $G$ is a Croke-Kleiner admissible group, acting geometrically on a CAT(0) space $X$, then a finitely generated subgroup of $G$ has finite height if and only if it is strongly quasi-convex. We also show that if $G \curvearrowright X$ is a flip CKA action then $G$ is quasi-isometric embedded into a finite product of quasi-trees. With further assumption on the vertex groups of the flip CKA action $G \curvearrowright X$, we show that $G$ satisfies property (QT) that is introduced by Bestvina-Bromberg-Fujiwara in \cite{BBF2}. 
\end{abstract}
\maketitle


\section{Introduction}
In \cite{CK02}, Croke and Kleiner study a particular class of graph of groups which they call {\it admissible groups}  and generalize fundamental groups of $3$--dimensional graph manifolds and torus complexes (see \cite{CK00}). If $G$ is an admissible group that acts geometrically on a Hadamard space $X$ then the action $G \curvearrowright X$ is called {\it Croke-Kleiner admissibe} (see Definition~\ref{defn:admissible})  termed by Guilbault-Mooney \cite{GM14}. The CKA action   is modeling on the JSJ structure of graph manifolds where the Seifert fibration is replaced by the following central extension of a general hyperbolic group:
\begin{equation}\label{centralExtEQ}
1\to Z(G_v)=\mathbb Z\to G_v\to H_v\to 1    
\end{equation}
However, CKA groups can encompass much more general class of groups and can actually serve as one of simplest algebraic means to produce interesting groups from any finite number of hyperbolic  groups. 

Let $\mathcal G$ be a finite graph  with $n$ vertices, each of which are associated with a hyperbolic group $H_i$. We then pick up an independent  set of primitive loxodromic elements in $H_i$ which   crossed with $\mathbb Z$ are the edge groups $\mathbb Z^2$. We identify $\mathbb Z^2$ in adjacent $H_i\times \mathbb Z$'s by flipping $\mathbb Z$ and loxodromic elements as did in flip
graph manifolds by Kapovich and Leeb \cite{KL98}. These are motivating examples of \textit{flip} CKA groups and actions, for the precise definition of flip CKA actions, we refer the reader to Section~\ref{subsection:qiehyp}.

The class of CKA actions has manifested a variety of interesting features in CAT(0) groups.  For instance, the equivariant visual boundaries  of  admissible actions are completely determined in  \cite{CK02}. Meanwhile, the non-homeomorphic  visual boundaries of torus complexes were constructed in \cite{CK00} and have sparked an intensive research on  boundaries of CAT(0) spaces. So far, the most of research on CKA groups is centered around the boundary problem (see \cite{GM14}, \cite{Gre16}).    In the rest of Introduction, we shall explain our results on the coarse geometry of Croke-Kleiner admissible groups and their subgroups.

\subsection{Proper actions on finite products of quasi-trees}
A \emph{quasi-tree} is a geodesic metric space quasi-isometric to a tree. 
Recently, Bestvina, Bromberg and Fujiwara \cite{BBF2} introduced a {\it (QT)} property for a finitely generated group:   $G$ acts properly on a finite product of quasi-trees so that the orbital map from  $G$ with word metrics is a quasi-isometric embedding. This is a stronger property of the finite asymptotic dimension by recalling that a quasi-isometric embedding implies finite asdim of $G$. It is known that Coxeter groups have property (QT) (see \cite{DJ99}), and thus every right-angled Artin group has property (QT) (see Induction~2.2 in \cite{BBF2}). Furthermore, the fundamental group of a compact special cube complex is undistorted in RAAGS (see \cite{HW08}) and then has property (QT). As a consequence, many 3-manifold groups have property (QT), among which we wish to mention  chargeless (including flip) graph manifolds \cite{HP15} and finite volume hyperbolic 3-manifolds \cite{Wise20}.  In \cite{BBF2}, residually finite hyperbolic groups and mapping class groups are proven to have property (QT). It is natural to ask which other groups have property (QT) rather than these groups above.

The main result of this paper adds flip CKA actions into the list of groups which have property (QT). The notion of  an \textit{omnipotent} group  is introduced by Wise in \cite{Wise00} and has found many applications in subgroup separability. We refer the reader to Definition \ref{defn:omnipotent} for its definition and note here that free groups \cite{Wise00}, surfaces groups \cite{B07}, and the more general class of virtually special hyperbolic groups \cite{Wise20} are omnipotent.


\begin{thm}
\label{thm:main4}
Let $G \curvearrowright X$ be a flip admissible action where for every vertex group the central extension  (\ref{centralExtEQ}) has omnipotent hyperbolic quotient group. Then $G$  acts properly on a finite product of quasi-trees so that the orbital map is a quasi-isometric embedding.
\end{thm}

\begin{rem}
It is an open problem whether every hyperbolic group is residually finite. 
In \cite[Remark 3.4]{Wise00}, Wise noted that if every hyperbolic group is residually finite, then any hyperbolic group is omnipotent.
\end{rem}


\begin{rem}
As a corollary, Theorem \ref{thm:main4} gives another proof that flip graph manifold groups have property (QT).  This was indeed one of motivations of this study (without noticing \cite{HP15}).

In \cite{HS13}, Hume-Sisto prove that the universal cover of any flip graph manifold is quasi-isometrically embedded in the product of three metric trees. However, it does not follow from their proof that the fundamental group of a flip graph manifold has property (QT).

\end{rem}

We now give an outline of the proof of Theorem~\ref{thm:main4} and explain some intermediate results, which we believe are of independent interest. 

\begin{prop}
\label{prop:main4}
Let $G \curvearrowright X$ be a flip CKA action. Then there exists a quasi-isometric embedding from $X$ to a  product $\mathcal X_1\times \mathcal X_2$ of two hyperbolic spaces. 

If $G_v=H_v\times Z(G_v)$ for every vertex $v\in T^0$ and $G_e=Z(G_{e_-})\times Z(G_{e_+})$ for every edge $e\in T^1$, then there exists a subgroup $\dot G<G$ of finite index at most $2$ such that the above Q.I. embedding is $\dot G$-equivariant.
\end{prop}

Let us describe briefly the construction of $\mathcal X\in \{\mathcal X_1, \mathcal X_2\}$. By Bass-Serre theory, $G$ acts on the Bass-Serre tree $T$ with vertex groups $G_v$ and edge groups $G_e$. Let $\mathcal V$ be one of the two sets of vertices in  $T$ with pairwise even distance. Note that  $G_v$ is the central extension of a hyperbolic group $H_v$ by $\mathbb Z$, so acts geometrically on a metric product $Y_v=\bar Y_v\times \mathbb R$ where $H_v$ acts geometrically on $\bar Y_v$ and $Z(G_v)$ acts by translation on $\mathbb R$-lines. Roughly, the space $\mathcal X$ is obtained by isometric gluing of  the boundary lines of $\bar Y_v$'s over vertices $v$ in the link of every $w \in T^0\setminus \mathcal V$. In proving Proposition \ref{prop:main4}, the main tool is the construction of a class of quasi-geodesic paths called \textit{special paths} between any two points in $X$.  See Section \ref{subsection:pathsystem} for the details and related discussion after  Theorem \ref{thm:generalization} below.   

To endow an action on $\mathcal X$, we pass to   an index at most 2 subgroup $\dot G$  preserving $\mathcal V$ and the stabilizer in $\dot G$ of $v\in \mathcal V$ is $G_v$ by Lemma \ref{lem:index2subgroup}. Under the  assumptions on $G_v$ and $G_e$'s, $\dot G$ acts by isometry on $\mathcal X$ and the Q.I. embedding is $\dot G$-equivariant.

To prove Theorem~\ref{thm:main4}, we exploit the  strategy  as \cite{BBF2} to produce a proper action on products of quasi-trees. By Lemma~\ref{YDistFormulaLem},  we first produce enough quasi-lines  $$\mathbb A=\cup_{v\in\mathcal V} \mathbb A_v$$ for the hyperbolic space $\mathcal X$ where $\mathbb A_v$ is a $H_v$-finite set of quasi-lines in $\bar Y_v$   so that the so-called distance formula follows in Proposition \ref{prop:distanceformulaX_1}. On the other hand, the ``crowd" quasi-lines in $\mathbb A$  may fail to satisfy the projection axioms in \cite{BBF}  with   projection constant required  from the distance formula. Thus,  we have to partition $\mathbb A$  into finite sub-collections of \textit{sparse} quasi-lines: $\forall \gamma\ne \gamma'$, $d(\gamma, \gamma')>\theta$ for a uniform constant $\theta$. 

Using local finiteness, we can partition  quasi-lines without respecting group action and prove the following general result. This generalize  the results of Hume-Sisto in \cite{HS13} to flip CKA actions.

\begin{thm}
\label{thm:main3}
Let $G \curvearrowright X$ be a flip CKA action. Then $G$ is quasi-isometric embedded into a finite product of quasi-trees. 
\end{thm}

However, the difficulty in establishing property (QT) is to partition all quasi-lines $\mathbb A=\cup_{i=1}^n\mathbb A_i$   so that each $\mathbb A_i$ is $\dot G$-invariant and sparse. In  \textsection \ref{SSConeoffSpace}, we cone  off the boundary lines of $Y_v$'s so that $\mathbb A_v$'s from different pieces $Y_v$ are ``isolated". The gives the   coned-off space $\dot {\mathcal X}$ with new distance formula in Proposition \ref{ConeoffDistFProp} so that $\mathcal X$ is quasi-isometric embedded into the product of $\dot {\mathcal X}$ with a quasi-tree from the boundary lines. See Proposition \ref{QIembedProp}.  

The goal is then to find a finite index subgroup $\ddot G<\dot G<G$ so that each orbit in $\mathbb A$ is sparse. This is done  in the following two steps:   

By residual finiteness of $H_v$, we first find a finite index subgroup $K_v<H_v$   whose orbit in $\mathbb A_v$  is  sparse. This follows the same argument in \cite{BBF2}.  Secondly, we need to reassemble those finite index subgroups $K_v$ as a finite index group $\ddot G$ so that the orbit in $\mathbb A$ is sparse. This step uses crucially   the omnipotence, with details   given in \textsection \ref{SSPartitonProperly}. The projection axioms thus fulfilled for each $\ddot G$-orbit produce a finite product of actions on quasi-trees, and finally, the   distance formula   finishes the proof of Theorem \ref{thm:main4}.


\subsection{Strongly quasi-convex subgroups}
The {\it height} of a finitely generated subgroup $H$ in a finitely generated group $G$ is the maximal $n \in \mathbb N$ such that there are distinct cosets $g_1 H,\dots,g_n H \in G/H$ such that the intersection
$g_1 H g_1^{-1} \cap \dots \cap g_n H g_n^{-1}$
is infinite. The subgroup $H$ is called {\it strongly quasi-convex} in $G$ if for any $L \ge 1$, $C \ge 0$ there exists $R = R(L,C)$ such that every $(L,C)$--quasi-geodesic in $G$ with endpoints in $H$ is contained in the $R$--neighborhood of $H$. We note that strong quasiconvexity does not depend on the choice of finite generating set of the ambient group and it agrees with quasiconvexity when the ambient group is hyperbolic. In \cite{GMRS}, the authors prove that quasi-convex subgroups in hyperbolic groups have finite height. It is a long-standing question asked by Swarup that whether or not the converse is
true (see Question 1.8 in \cite{Bes}). Tran in \cite{Tra19} generalizes the result of \cite{GMRS} by showing that strongly quasi-convex subgroups in any finitely generated group have finite height. It is natural to ask whether or not the converse is true in this setting (i.e, finite height implies strong quasiconvexity). If the converse is true, then we could characterize strongly quasi-convex subgroup of a finitely generated group purely in terms
of group theoretic notions.

In \cite{NTY19}, the authors prove that having finite height and strong quasiconvexity are  equivalent  for all finitely generated $3$--manifold groups  except the only ones containing the Sol command in its sphere-disk decomposition, and the graph manifold case  was an essential case treated there. More precisely, Theorem~1.7 in \cite{NTY19} states that finitely generated subgroups of the fundamental group of a graph manifold are strongly quasi-convex if and only if they have finite height. The second main result of this paper is to  generalize this result to Croke-Kleiner admissible action $G \curvearrowright X$.

\begin{thm}
\label{thm:generalization}
    Let $G \curvearrowright X$ be a CKA action. Let $K$ be a nontrivial, finitely generated infinite index subgroup of $G$. Then the following are equivalent.
\begin{enumerate}
    \item
     \label{thm1:item1}$K$ is strongly quasi-convex.
    \item
    \label{thm1:item2}$K$ has finite height in $G$.
    \item
    \label{thm1:item3} $K$ is virtually free and every infinite order elements are Morse.
    \item
    \label{thm1:item4}Let $G \curvearrowright T$ be the action of $G$ on the associated Bass-Serre tree. $K$ is virtually free and   the action of $K$ on the tree $T$ induces a quasi-isometric embedding of $K$ into $T$.
\end{enumerate}
\end{thm}
We prove Theorem~\ref{thm:generalization} by showing that $(\ref{thm1:item1}) \implies (\ref{thm1:item2}) \implies (\ref{thm1:item3}) \implies (\ref{thm1:item1})$ and $(\ref{thm1:item3}) \Leftrightarrow (\ref{thm1:item4})$. Similarly as in \cite{NTY19}, the heart part of Theorem~\ref{thm:generalization} is the implication $(\ref{thm1:item3}) \implies (\ref{thm1:item1})$. We briefly review ideas in the proof of Theorem~1.7 in \cite{NTY19}. Suppose that $K$ is a finitely generated finite height subgroup of $\pi_1(M)$ where $M$ is a graph manifold. Let $M_K \to M$ be the covering space of $M$ corresponding to $K$. The authors in \cite{NTY19} prove that $K$ is strongly quasi-convex in $\pi_1(M)$ by using Sisto's notion of path system $\mathcal{PS}(\tilde{M})$ in the universal cover $\tilde{M}$ of $M$, and prove that the preimage of the Scott core of $M_K$ in $\tilde{M}$ is $\mathcal{PS}(\tilde{M})$--contracting in the sense of Sisto. In this paper, the strategy of the proof of Theorem~\ref{thm:generalization} is similar to the proof of Theorem~1.7 in \cite{NTY19} where we still use Sisto's  path system in $X$ but details are different. Sisto's construction of \textit{special paths} are carried out only in flip graph manifolds. Our construction of $(X, \mathcal{PS}(X))$ relies on the work of Croke-Kleiner \cite{CK02} and applies to any  admissible space $X$ (so   any nonpostively curved graph manifold). We then    construct a subspace $C_K \subset X$ on which $K$ acts geometrically  and show that $C_K$ is contracting in $X$ with respect to the path system $(X, \mathcal{PS}(X))$. As a consequence, $K$ is strongly quasi-convex in $G$.

To conclude the introduction, we list a few questions and problems. 

Quasi-isometric classification of graph manifolds has been studied by Kapovich-Leeb \cite{KL98} and a complete quasi-isometric classification for fundamental
groups of graph manifolds is given by Behrstock-Neumann \cite{BN08}. Kapovich-Leeb prove that for any graph manifold $M$, there exists a flip graph manifold $N$ such that their fundamental groups are quasi-isometric. We would like to know that whether or not such a result holds for admissible groups.
\begin{ques}
Let $G$ be an admissible group such that each vertex group is the central extension of a omnipotent hyperbolic CAT(0) group by $\Z$. Does there exist flip CKA action $G' \curvearrowright X$ so that $G$ and $G'$ are quasi-isometric? 
\end{ques}

\begin{ques}[Quasi-isometry rigidity]
Let $G \curvearrowright X$ be a flip CKA action, and $Q$ be a finitely generated group which is quasi-isometric to $G$. Does there exist a finite index subgroup $Q'< Q$ such that $Q'$ is a flip CKA group?
\end{ques}

With a positive answer to the above questions, we hope one can try to follow the strategy described in \cite{BN08} to attack the following. 

\begin{prob}
Under the assumption of Theorem \ref{thm:main4}, give a quasi-isometric classification of admissible actions.   
\end{prob}

In \cite{Liu13},  Liu  showed that the fundamental group of a non-positively curved graph manifold $M$ is virtually special (the case $\partial M \neq \emptyset$ was also obtained independently by Przytycki–Wise \cite{PW14}). Thus, it is natural to ask the following.

\begin{ques}
Let $G \curvearrowright X$ be a CKA action where vertex groups are the central extension  of  a virtually special hyperbolic group by the integer group. Is $G$ virtually special?
\end{ques}

As above, a positive answer to the question (with virtual compact specialness) would give an other proof of  Theorem \ref{thm:main4} under the same assumption.


\subsection*{Overview}
In Section~\ref{section:preliminary}, we review some concepts and results about Croke-Kleiner admissible groups. In Section~\ref{subsection:pathsystem}, we construct special paths in admissible spaces and give some results that will be used in the later sections. The proof of Theorem~\ref{thm:main3} and Proposition~\ref{prop:main4} is given in Section~\ref{qiefinitequasitrees}. We prove Theorem~\ref{thm:main4} and Theorem~\ref{thm:generalization} in Section~\ref{sec:properaction} and Section~\ref{sec:finiteheightquasiconvex} respectively.

\subsection*{Acknowledgments}
We would like to thank Chris Hruska, Hongbin Sun and Dani Wise for helpful conversations. W. Y. is supported by the National Natural Science Foundation of China (No. 11771022).

\section{Preliminary}
\label{section:preliminary}
\emph{Admissible groups} firstly introduced in \cite{CK02}. This is a particular class of graph of groups that includes fundamental groups of $3$--dimensional graph manifolds (i.e, compact $3$--manifolds are obtained by gluing some circle bundles). In this section, we review  admissible groups and their properties that will used throughout in this paper.
\begin{defn}
\label{defn:admissible}
A graph of group $\mathcal{G}$ is \emph{admissible} if
\begin{enumerate}
    \item $\mathcal{G}$ is a finite graph with at least one edge.
    \item Each vertex group ${ G}_v$ has center $Z({ G}_v) \cong \Z$, ${ H}_v \colon = { G}_{v} / Z({ G}_v)$ is a non-elementary hyperbolic group, and every edge subgroup ${ G}_{e}$ is isomorphic to $\Z^2$.
    \item Let $e_1$ and $e_2$ be distinct directed edges entering a vertex $v$, and for $i = 1,2$, let $K_i \subset { G}_v$ be the image of the edge homomorphism ${G}_{e_i} \to {G}_v$. Then for every $g \in {\bar G}_v$, $gK_{1}g^{-1}$ is not commensurable with $K_2$, and for every $g \in  G_v \setminus K_i$, $gK_ig^{-1}$ is not commensurable with $K_i$. 
    \item For every edge group ${ G}_e$, if $\alpha_i \colon { G}_{e} \to { G}_{v_i}$ are the edge monomorphism, then the subgroup generated by $\alpha_{1}^{-1}(Z({ G}_{v_1}))$ and $\alpha_{2}^{-1}(Z({ G}_{v_1}))$ has finite index in ${ G}_e$.
\end{enumerate}
A group $G$ is \emph{admissible} if it is the fundamental group of an admissible graph of groups.
\end{defn}

\begin{defn}
\label{defn:admissibleaction}
We say that the action $G \curvearrowright X$ is an {\it Croke-Kleiner admissible} (CKA) if $G$ is an admissible group, and $X$ is a Hadamard space, and the action is geometrically (i.e, properly and cocompactly by isometries)
\end{defn}

Examples of admissible actions:
\begin{enumerate}
    \item Let $M$ be a nongeometric graph manifold that admits a nonpositively curve metric. Lift this metric to the universal cover $\tilde{M}$ of $M$, and we denote this metric by $d$. Then the action $\pi_1(M) \curvearrowright (\tilde{M}, d)$ is a CKA action.
    \item Let $T$ be the torus complexes constructed in \cite{CK00}. Then $\pi_1(T) \curvearrowright \tilde{T}$ is a CKA action.
    \item Let $H_1$ and $H_2$ be two torsion-free hyperbolic groups such that they act geometrically on $CAT(0)$ spaces $X_1$ and $X_2$ respectively. Let $G_i = H_i \times \Z$ (with $i =1,2$), then $G_i$ acts geometrically on the $CAT(0)$ space $Y_ i = X_{i} \times \R$. A primitive hyperbolic element in $H_i$ gives a totally geodesic torus $T_i$ in the quotient space $Y_{i}/G_i$. Choose a basis on each torus $T_i$. Let $f \colon T_1 \to T_2$ be a flip map.  Let $M$ be the space obtained by gluing $Y_1$ to $Y_2$ along the homemorphism $f$. We note that there exists a metric on $M$ such that with respect to this metric, $M$ is a locally $CAT(0)$ space. Then $G \curvearrowright \tilde M$ is a CKA action.
\end{enumerate}

Let $G \curvearrowright X$ be an admissible action, and let $G \curvearrowright T$ be the action of $G$ on the associated Bass-Serre tree. Let $T^0 = Vertex(T)$ and $T^1 = Edge(T)$ be the vertex and edge sets of $T$. For each $\sigma \in T^0 \cup T^1$, we let $G_{\sigma} \le G$ be the stabilizer of $\sigma$. For each vertex $v \in T^0$, let $Y_{v} := Minset(Z(G_v)) := \cap_{g \in Z(G_v)} Minset(g)$ and for every edge $e \in E$ we let $Y_{e} := Minset(Z(G_e)) := \cap_{g \in Z(G_e)} Minset(g)$.
We note that the assignments $v \to Y_v$ and $e \to Y_e$ are $G$--equivariant with respect to the natural $G$ actions.

 The following lemma is well-known.
 \begin{lem}
If $H=\mathbb Z^k$ for some $k\ge 1$ then $Minset(H)=\cap_{h\in H} Minset(h)$ splits isometrically as a metric product $C\times \mathbb E^k$ so that $H$ acts trivially on $C$ and as a translation lattice on $\mathbb E^k$. Moreover, $Z(H,G)$ acts cocompactly on $C\times \mathbb E^k$.
\end{lem}
As a corollary, we have  

\begin{enumerate}
    \item $G_v$ acts co-compactly on $Y_v=\bar Y_v\times \mathbb R$ and $Z(G_v)$ acts by translation on the $\mathbb R$--factor and trivially on $\bar Y_v$ where $\bar Y_v$ is a Hadamard space.
    \item $G_e=\mathbb Z^2$ acts co-compactly on $Y_e=\bar Y_e \times \mathbb R^2\subset Y_v$ where $\bar Y_e$ is a compact Hadamard space. 
    
    \item if $\langle t_1\rangle=Z(G_{v_1}), \langle t_2\rangle=Z(G_{v_2})$ then $\langle t_1, t_2\rangle$ generates a finite index subgroup of $G_e$.  
\end{enumerate}

We summarize results in Section~3.2 of \cite{CK02} that will be used in this paper. 
\begin{lem}
\label{defn:vertexedgespace}
Let $G \curvearrowright X$ be an CKA action. Then there exists a constant $D >0$ such that the following holds.
\begin{enumerate}
    \item $\cup_{v \in T^0} {N}_{D}(Y_v) = \cup_{e \in T^1} {N}_{D}(Y_e) = X$. We define $X_v := {N}_{D}(Y_v)$ and $X_e := {N}_{D}(Y_e)$ for all $v \in T^0$, $e \in T^1$.
    \item If $\sigma, \sigma' \in T^0 \cup T^1$ and $X_{\sigma} \cap X_{\sigma'} \neq \emptyset$ then $d_{T}(\sigma, \sigma') < D$.
\end{enumerate}
\end{lem}

{\bf Strips in admissible spaces:} (see Section~4.2 in \cite{CK02}).
 We first choose, in a $G$--equivariant way, a plane $F_e \subset Y_e$ for each edge $e \in T^1$. 
Then for every pair of adjacent edges $e_1$, $e_2$. we choose, again equivariantly, a minimal geodesic from $F_{e_1}$ to $F_{e_2}$; by the convexity of $Y_v = \bar{Y}_{v} \times \R$, $v:= e_{1} \cap e_{2}$, this geodesic determines a Euclidean strip $\mathcal{S}_{e_1e_2} : = \gamma_{e_1e_2} \times \R$ (possibly of width zero) for some geodesic segment $\gamma_{e_1e_2} \subset \bar{Y}_v$. Note that $\mathcal{S}_{e_1e_2} \cap F_{e_i}$ is an axis of $Z(G_v)$. Hence if $e_1, e_2, e \in E$, $e_{i} \cap e = v_{i} \in V$ are distinct vertices, then the angle between the geodesics $\mathcal{S}_{e_1e} \cap F_{e}$ and $\mathcal{S}_{e_2e} \cap F_{e}$ is bounded away from zero.

\begin{rem}
\begin{enumerate}
\item We note that it is possible that  $\gamma_{e_1,e_2}$ is just a point. The lines  $S_{e_1, e_2} \cap F_{e_1}$ and $S_{e_1, e_2} \cap F_{e_2}$ are axes of $Z(G_v)$.
\item There exists a uniform constant such that for any edge $e$, the Hausdorff distance between two spaces $F_{e}$ and $X_{e}$ is no more than this constant.
\end{enumerate}
\end{rem}

\begin{rem}
\label{rem:indexfunction}
There exists a $G$--equivariant coarse $L$--Lipschitz map  $\rho \colon X \to T^0$  such that  $x \in X_{\rho(x)}$ for all $x \in X$. The map $\rho$ is called {\it indexed map}. We refer the reader to Section~3.3 in \cite{CK02} for existence of such a map $\rho$.

 \end{rem}

\begin{defn}[Templates, \cite{CK02}]
\label{defn:template}
A {\it template} is a connected Hadamard space $\mathcal{T}$ obtained from disjoint collection of Euclidean planes $\{W\}_{W \in Wall_{\mathcal{T}}}$ (called {\it walls}) and directed Euclidean strips $\{\mathcal{S}\}_{\mathcal{S} \in Strip_{\mathcal{T}}}$ (a direction for a strip $\mathcal{S}$ is a direction for its $R$--factor $\mathcal{S} \simeq I \times \R$) by isometric gluing  subject to the following conditions.
\begin{enumerate}
    \item The boundary geodesics of each strip $\mathcal{S} \in Strip_{\mathcal{T}}$ , which we will refer to as {\it singular geodesics}, are glued isometrically to distinct walls in $Wall_{\mathcal{T}}$.
    \item Each wall $W \in Wall_{\mathcal{T}}$ is glued to at most two strips, and the gluing lines are not parallel.
 
\end{enumerate}
\end{defn}

{\bf Notations: }
We use the notion $a \preceq_{K} b$ if the exists $C = C(K)$ such that $a \le Cb + C$, and we use the notion $a \sim_{K} b$ if $a \preceq_{K} b$ and $b \preceq_{K} a$. Also, when we write $a \asymp_{K} b$ we mean that $a/C \le b \le Ca$.

Denote   by $Len^1(\gamma)$ and  $Len(\gamma)$ the $L^1$--length  and $L^2$--length of a path $\gamma$ in a metric product space  $A \times B$. These two lengths are equal for a path  if it  is parallel to a factor; in general, they are bilipschitz.

\section{Special paths in CKA action $G \curvearrowright X$}
\label{subsection:pathsystem}
Let $G \curvearrowright X$ be a CKA action.
In this section, we are going to define \emph{special paths} (see Definition~\ref{SpecialPathDefn}) in $X$ that will be used on the latter sections. Roughly speaking, each special path in $X$ is a concatenation of geodesics in consecutive pieces $Y_v$'s of $X$ and they are uniform quasi-geodesic in the sense that there exists a constant $\mu = \mu(X)$ such that every special path is $(\mu, \mu)$--quasi-geodesic.

We first introduce the class of \emph{special paths} in a template which shall be mapped to \emph{special paths} in $X$   up to a finite Hausdorff distance. 

\subsection{Special paths in a template}

\begin{defn}
Let $\mathcal{T}$ be the template given by Definition~\ref{defn:template}.
A (connected) path $\gamma$ in $\mathcal T$ is called \emph{special path} if $\gamma$ is a concatenation $\gamma_0\gamma_1\cdots\gamma_n$ of geodesics $\gamma_i$ such that each $\gamma_i$ lies on the strip  $\mathcal S_i$ adjacent to  $W_{i}$ and $W_{i+1}$. 
\end{defn}

\begin{rem}
By the construction of the template, the endpoints of $\gamma_i$ $(1\le i<n)$ must be the intersection points of singular geodesics on walls $W_{i}, W_{i+1}$. 
\end{rem}

We use Lemma~\ref{TemplateSPathLem} in the proof of Proposition~\ref{prop:spepathisqg}.
\begin{lem}\label{TemplateSPathLem}
Assume that the angles between the singular geodesics on walls are   between   $\beta$ and $\pi-\beta$ for a universal constant $\beta\in (0, \pi)$. There exists a constant $\mu\ge 1$ such that any special path is a $(\mu, 0)$--quasi-geodesic. 
\end{lem}
\begin{proof}
Let $\gamma$ be a special path with endpoints $x, y$. We are going to prove that $Len(\gamma)\le \mu d(x,y)$ for a constant $\mu\ge 1$. Since any subpath of a special path is special, this proves the conclusion.

Let $\alpha$ be the unique CAT$(0)$ geodesic between $x$ and $y$. By the construction of the template, if $\alpha$ does not pass through the intersection point $z_W$ of the singular geodesics on a wall $W$, then it passes through a point $x_W$ on one singular geodesic $L_-$ and then a point $y_W$ on the other singular geodesic $L_+$. Recall that the angle between the singular geodesics on walls are uniformly between $\beta$ and $\pi-\beta$. There exists a constant  $\mu\ge 1$   depending on $\beta$ only such that $ d(x_W, z_W)+d(y_W,z_W)\le \mu d(x_W, y_W)$. We thus replace $[x_W,y_W]$ by $[x_W,z][z,y_W]$ for every possible triangle $\Delta(x_Wy_Wz_W)$ on each wall $W$. The resulted path then connects consecutively the points $z_W$ on the walls $W$ in the order of their intersection with $\alpha$, so  it is the special path $\gamma$ from $x$ to $y$ satisfying the following inequality  
$$
Len(\gamma) \le \mu d(x, y)
$$
Thus, we proved that $\gamma$ is a $(\mu, 0)$--quasi-geodesic. 
\end{proof}

We are going to define a template associated to a geodesic in the Bass-Serre tree  as the following.

\begin{defn}[Standard template associated to a geodesic $\gamma \subset T$]
\label{defn:standardtemplate}
Let $\gamma$ be a geodesic segment in the Bass-Serre tree $T$. We begin with a collection of walls $W_e$ and an isometry $\phi_e \colon W_e \to F_e$ for each edge $e \subset \gamma$. For every pair $e$, $e'$ of adjacent edges of $\gamma$, we let $\mathcal{\hat S}_{e,e'}$ be a strip which is isometric to $\mathcal{S}_{e,e'}$ if the width of  $\mathcal{S}_{e,e'}$ is at least $1$, and isometric to $[0,1] \times \R$ otherwise; we let $\phi_{e,e'} \colon \mathcal{\hat S}_{e,e'} \to \mathcal{S}_{e,e'}$ be an affine map which respects product structure ($\phi_{e,e'}$ is an isometry if the width of $\mathcal{S}_{e,e'}$ is greater than or equal to $1$ and compresses the interval otherwise). We construct $\mathcal{T}_{\gamma}$ by gluing the strips and walls so that the maps $\phi_e$ and $\phi_{e,e'}$ descend to continuous maps on the quotient, we denote the map from $\mathcal{T}_{\gamma} \to X$ by $\phi_{\gamma}$.
\end{defn}

The following lemma is cited from Lemma~4.5 and Proposition~4.6 in \cite{CK02}.
\begin{lem}
\label{lem:qie}
\begin{enumerate}
    \item There exists $\beta = \beta(X) >0$ such that the following holds. For any geodesic segment $\gamma \in T$, the angle function $\alpha_{\gamma} \colon Wall^{o}_{\mathcal{T}} \to (0, \pi)$ satisfies $0 < \beta  \le \alpha_{\gamma} \le \pi - \beta < \pi$.
    \item There are constants $L, A >0$ such that the following holds. Let $\gamma$ be a geodesic segment in $T$, and let $\phi_{\gamma} \colon \mathcal{T}_{\gamma} \to X$ be the map given by Definition~\ref{defn:standardtemplate}. 
    Then $\phi_{\gamma}$ is a $(L,A)$--quasi-isometric embedding. Moreover, for any $x, y \in [\cup_{e \subset \gamma} X_{e}] \cup [\cup_{e,e' \subset \gamma} \mathcal{S}_{e,e'}]$, there exists a continuous map $\alpha \colon [x,y] \to \mathcal{T}$ such that $d(\phi_{\gamma} \circ \alpha , id|_{[x,y]}) \le L$.
\end{enumerate}
\end{lem}

\subsection{Special paths in the admissible space $X$}
In this subsection, we are going to define special paths in $X$.

Recall  that we choose a $G$--equivariant family of Euclidean planes $\{F_e: F_e\subset Y_e\}_{e \in T^1}$. For every pair of planes $(F_e, F_{e'})$ so that  $v = e \cap e'$,  a minimal geodesic between $F_e, F_{e'}$   in $Y_v$  determines a strip $\mathcal{S}_{ee'} = \gamma_{ee'} \times \R$ for some geodesic $\gamma_{ee'}\subset \bar Y_v$. It is possible that $\gamma_{ee'}$ is trivial so the width of the strip is zero.
Let   $x\in X_v$ and $e$ an edge  with an endpoint $v$. The minimal geodesic from $x$ to $F_e$ (possibly not belong to $Y_v$)  also define a strip $\mathcal{S}_{xe} = \gamma_{xe} \times \R$   where the geodesic $\gamma_{xe}\subset \bar Y_v$ is the projection to $\bar Y_v$ of the intersection of this minimal geodesic with $Y_v$. Thus, $x$ is possibly not   in the strip $\mathcal{S}_{xe}$ but within its $D$-neighborhood by Lemma \ref{defn:vertexedgespace}. 






\begin{defn}[Special paths in $X$]
\label{SpecialPathDefn}
Let $\rho \colon X \to T^0$ be the indexed map
given by Remark~\ref{rem:indexfunction}.
Let $x$ and $y$ be two   points in $X$. If $\rho(x) = \rho(y)$ then we define a \emph{special path} in $X$ connecting $x$ to $y$ is the geodesic $[x,y]$. Otherwise, let $e_{1} \cdots e_{n}$ be the geodesic edge path connecting  $\rho(x)$ to $\rho(y)$ and let $p_i\in F_{e_i}$ be the intersection point of the strips $\mathcal S_{e_{i-1}e_i}$ and $\mathcal S_{e_{i}e_{i+1}}$, where $e_{0}:=x$ and $e_{n+1}:=y$.  The \textit{special path} connecting $x$ to $y$ is the concatenation of the geodesics $$[x, p_{1}][p_{1}, p_{2}]\cdots [p_{{n-1}}, p_{n}][p_{n}, y].$$
\end{defn}

\begin{rem}
\label{rem:independent}
By definition, the special path except the $[x, p_{1}]$ and $[p_{n}, y]$ depends only on the geodesic $e_{1} \cdots e_{n}$ in $T$ and the choice of planes $F_e$.
\end{rem}

\begin{figure}[htb] 
\centering \scalebox{0.6}{
\includegraphics{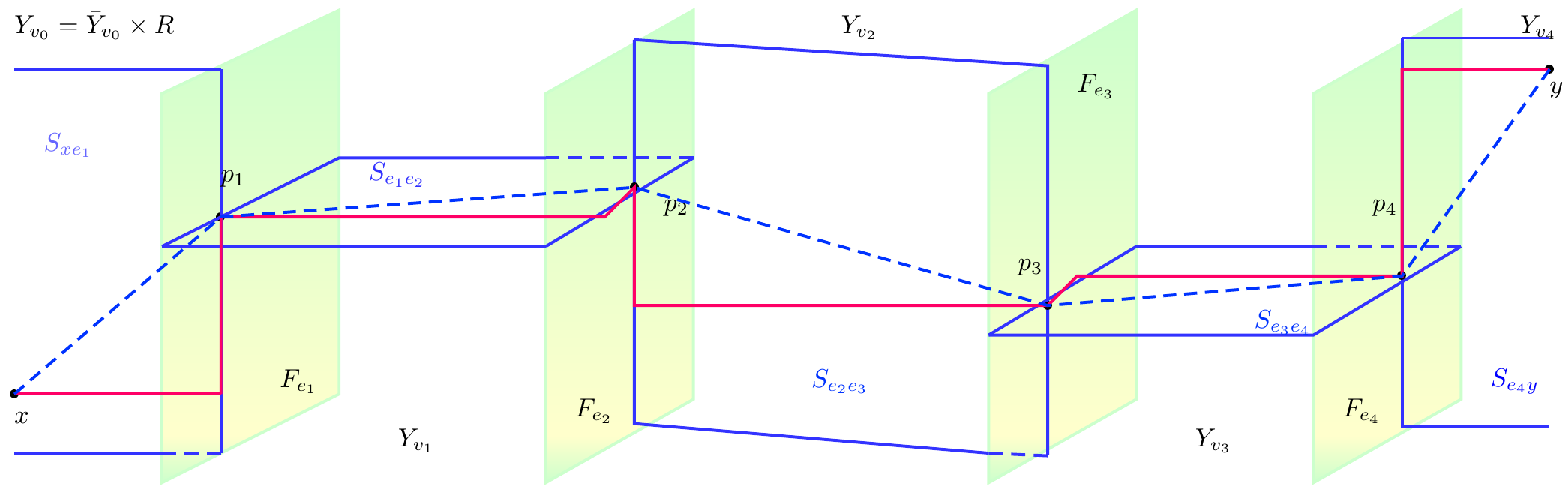} 
} \caption{Special path $\gamma$: the dotted and blue path from $x$ to $y$} \label{figure1}
\end{figure}

\subsection{Special paths in the admissible space $X$ are uniform quasi-geodesic}
In this section, we are going to prove the following proposition.
\begin{prop}
 \label{prop:spepathisqg}
There exists a constant $\mu >0$ such that every special path $\gamma$ in $X$ is a $(\mu, \mu)$--quasi-geodesic.
\end{prop}
To get into the proof of Proposition~\ref{prop:spepathisqg}, we need several lemmas (see Lemma~\ref{ConnectivityLem} and Lemma~\ref{HSlideLem}).

\begin{lem}\cite[Lemma~3.17]{CK02}
\label{ConnectivityLem}
There exists a constant $C>0$ with the following property. Let $[x,y]$ be a geodesic in $X$ with $\rho(x)\ne \rho(y)$ and $e_{1} \cdots e_{n}$ be the geodesic edge path connecting  $\rho(x)$ to $\rho(y)$. Then there exists a sequence of points $z_i\in [x,y]\cap N_C(F_{e_i})$ such that $d(x,z_i)\le d(x,z_j)$ for any $i\le j$.  
\end{lem}

Let $[x,y]$ be a geodesic in $Y_v=\bar Y_v\times \mathbb R$ with $y\in F_e$. 
We apply a \textit{minimizing horizontal slide} of the endpoint $y \in F_e$ to obtain a point $z\in F_e$ so that $[y,z]$ is parallel to $\bar Y_v$ and the projection of $[x,z]$ on $\bar Y_v$ is orthogonal to $F_e\cap \bar Y_v$.

\begin{lem}\label{HSlideLem}
Let $x\in X_{v_0}, y\in X_{v_n}$ where $v_0, v_n$ are the endpoints of a geodesic $e_1\cdots e_n$ in $T$. Then there exists a universal constant $C>0$ depending on $X$ such that for each $1\le i\le n$, we have
$$d(x, p_i)+d(p_i, y)\le C d(x, y) +C$$ where $p_i=\mathcal S_{e_{i-1}e_i}\cap \mathcal S_{e_{i}e_{i+1}}$ and $e_{0}:=x$ and $e_{n+1}:=y$.
\end{lem}
\begin{proof}
We use the notion $a \asymp b$ if there exists $K = K(X)$ such that $a/ K \le b \le Ka$.

Let $D$ be the constant given by Lemma~\ref{defn:vertexedgespace} and satisfying Lemma \ref{ConnectivityLem}   such that  $X_v= {N}_{D}(Y_{v})$. Without loss of generality, we can assume $x\in Y_{v_0}, y\in Y_{v_n}$.

Denote $e=e_i$ and $e^-=e_{i-1}, e^+=e_{i+1}$ in this proof. By Lemma \ref{ConnectivityLem}, there exists a point $q\in F_e$ such that $d(q,[x,y])\le D$. For ease of computation, we will consider the \textit{mixed length} $||x- q||_1+d(q,y)$ of the path $[x,q][q, y]$  which satisfies
\begin{equation}\label{MixedLenEQ}
||x- q||_1+d(q,y) \asymp d(x,q)+d(q,y) \le d(x,y)+2D
\end{equation}
where $||\cdot||_1$ is the $L^1$--metric on the metric product $Y_v=\bar Y_{v^-}\times \mathbb R$.

Note that the Euclidean plane $F_e \subset Y_v\cap Y_{v'}$  for $e=\overline{vv'}$ contains two non-parallel lines $l_e^-:=\mathcal{S}_{e^-e} \cap F_e$ and  $l_{e}^+:=\mathcal{S}_{ee^+} \cap F_e$. So we can apply a  minimizing horizontal slide of the endpoint $q$ of $[x,q]$ in $Y_v$ to a point $z$ on $l_e^-$.  On the one hand, since the line $l_e^-$   on $F_e$ is $\mathbb R$--factor of $Y_e$, this slide decreases the $L_1$-distance $||x- p||_1$ by $d(q,z)+C$ for a constant $C$ depending on hyperbolicity constant of  $\bar Y_v$. On the other hand, by the triangle inequality,  this slide increases $d(q, y)$ by at most $d(q ,z)$. Hence, we obtain 
$$
|(||x- q||_1+d(q,y)) - (||x- z||_1+d(z,y))| \le C
$$   Similarly, by a  minimizing horizontal slide of the endpoint $z$ of $[z,y]$ in $Y_{e'}$  to $p$,  
$$
|(||x- p||_1+d(p,y)) - (||x- z||_1+d(z,y))| \le C
$$
yielding
$$
|(||x- q||_1+d(q,y)) - (||x- p||_1+d(p,y))| \le 2C
$$
Together with (\ref{MixedLenEQ}) this completes the proof of the lemma. 
\end{proof}


\begin{proof}[Proof of Proposition~\ref{prop:spepathisqg}]
Let  $\gamma$ be the special path from $x$ to $y$ for  $x, y\in X$ so that $\rho(x)\ne \rho(y)$; otherwise it is a geodesic, and thus there is nothing to do. Let $e_1\cdots e_n$ be the geodesic in $T$ from $\rho(x)$ to $\rho(y)$. With notations as above (see Definition~\ref{SpecialPathDefn}), 
$$\gamma=[x, p_{1}][p_{1}, p_{2}]\cdots [p_{{n-1}}, p_{n}][p_{n}, y]$$  
By Lemma \ref{HSlideLem},  there exists a constant $C\ge 1$ such that   
$$
\begin{array}{rl}
 &d(x, p_{1})+d(p_{1}, p_{n})+d(p_{n},y)\\
 \le & C d(x, p_{1})+ C d(p_{1}, y)+C\\ 
 \le & C^2d(x,y)+C^2+C
\end{array}
$$
Denoting $\alpha=[p_{1}, p_{2}]\cdots [p_{n-1}, p_{n}]$, it remains to give a linear bound on $
\ell(\alpha)$ in terms of $d(p_1, p_n)$.

By Lemma~\ref{lem:qie}, there exists a $K$--template $(\mathcal T, f, \phi)$ for the $e_1\cdots e_n$ such that $\phi$ is a $(L,A)$--quasi-isometric map from the template $\mathcal T$ to the union of the planes $\{F_{e_i}: 1\le i\le n\}$ with the strips $\{\mathcal S_{e_{i-1}e_i}: 1\le i\le n\}$. Moreover, $\phi$ sends walls and strips of $\mathcal T$ to the $K$--neighborhood of planes $F_{e_i}$ and  strips $\mathcal S_{e_{i-1}e_i}$ of  $\mathcal T$ accordingly. Hence, $\phi$ maps the intersection point on $W_{e_i}$ of the singular geodesics of two strips in $\mathcal T$ to  a finite $K$--neighborhood of  $p_i$ $(1\le i\le n)$. Since the map $\phi$ is affine on strips and isometric on walls of $\mathcal T$, we conclude that there exists a special path $\tilde \alpha$ in $\mathcal T$ such that $\phi(\tilde \alpha)$ is sent to a finite neighborhood of the special path $\alpha$.   Lemma \ref{TemplateSPathLem} then implies that $\tilde \alpha$ is a $(C_1, C_1)$--quasi-geodesic for some $C_1>1$ so $\alpha$ is  a $(\mu, \mu)$--quasi-geodesic for some $\mu$ depending on $L, A, K, C_1$. The proof  is complete.
\end{proof}

\section{Quasi-isometric embedding  of admissible groups into product of trees}
\label{qiefinitequasitrees}
A \emph{quasi-tree} is a geodesic metric space quasi-isometric to a tree.
In this section, we are going to prove Theorem~\ref{thm:main3} that states if $G \curvearrowright X$ is a {\it flip CKA action} (see Definition~\ref{defn:flipadaction}) then $G$ is quasi-isometric embedded into a finite product of quasi-trees. The strategy is that we first show that the space $X$ is quasi-isometric embedded into product of two hyperbolic spaces $\mathcal{X}_1$, $\mathcal{X}_2$ (see Subsection~\ref{subsection:qiehyp}). We then show that each hyperbolic space $\mathcal{X}_i$ is quasi-isometric embedded into a finite product of quasi-trees (see Subsection~\ref{subsection:qieproducttrees}).

\subsection{Flip CKA actions and constructions of two hyperbolic spaces}
\label{subsection:flipactions}
Let $G \curvearrowright X$ be a CKA action.
Recall that each $Y_v$ decomposes as a metric product of a hyperbolic Hadamard space $\bar Y_v$ with the real line $\mathbb R$ such that $\bar Y_v$ admits a geometric action of $H_v$.  Recall  that we choose a $G$--equivariant family of Euclidean planes $\{F_e: F_e\subset Y_e\}_{e \in T^1}$.

\begin{defn}[Flip CKA action]
\label{defn:flipadaction}

If for each edge $e:=[v,w]\in T^1$,  the   boundary line  $\ell=\bar Y_v \cap F_e$ is parallel  to the $\R$--line in $Y_w = \bar{Y}_w \times \R$, then the CKA action is called \textit{flip} in sense of Kapovich-Leeb.  
 
\end{defn}

Let  $\mathbb L_v$ be the set of \textit{boundary lines} of $\bar Y _v$ which are intersections of $\bar Y_v$ with $F_e$ for all edges $e$ issuing from $v$. Thus, there is a canonical one-to-one correspondence between $\mathbb L_v$ and the link  of $v$ denoted by $Lk(v)$. 

 
\begin{defn}
A \textit{flat link} is the countable union of (closed) flat strips of width 1 glued along a common boundary line called the \textit{binding line}. 
\end{defn}

{\bf  Construction of hyperbolic spaces $\mathcal{X}_1$ and $\mathcal{X}_2$:}
We first partition the vertex set $T^0$ of the Bass-Serre tree into two {disjoint} class of vertices $\mathcal V_1$ and $\mathcal V_2$ such that   if $v$ and $v'$ are in $\mathcal{V}_i$ then $d_{T}(v,v')$ is even.

Given  $\mathcal{V} \in \{\mathcal V_1, \mathcal V_2\}$, we shall build a geodesic (non-proper) hyperbolic space $\mathcal X$ by glueing $\bar Y_v$ for all $v\in \mathcal V$ along the boundary lines via  flat links. 


Consider the set of vertices  in $\mathcal V$ such that their pairwise distance in $T$  equals 2.   Equivalently, it is the union of the links of every  vertex $w\in T^0 \setminus \mathcal{V}$.  For any $v_1\ne v_2\in Lk(w)$, the edges $e_1=[v_1, w]$ and $e_2=[v_2, w]$ determine two corresponding boundary lines $\ell_1 \in \mathbb L(v_1)$ and $\ell_2\in \mathbb L(v_2)$ which are the intersections of $\bar Y_{v_i}$ with $F_{e_i}$ for $i=1, 2$  respectively. There exists a canonical identification between $\ell_1$ and $\ell_2$ so that their $\mathbb R$--coordinates equal in the metric product $Y_w=\bar Y_w\times \mathbb R$.   

Note the link $Lk(w)$ determines  a flat link $Fl(w)$ so that the flat strips are one-to-one correspondence with $Lk(w)$.  In equivalent terms, it is a metric product   $Lk(w)\times \mathbb R$, where $\mathbb R$ is parallel to the binding line.  

For each $w\in T^0\setminus \mathcal{V}$, the set of hyperbolic spaces  $\bar Y_{v}$ where $v\in Lk(w)$ are glued to the flat links $Fl(w)$ along the boundary lines of flat strips and of hyperbolic spaces with the identification indicated above. Therefore, we obtain a metric space $\mathcal X$   from the   union of $\{\bar Y_v: v\in \mathcal{V}\}$ and flat links $\{Fl(w), w\in T^0 \setminus \mathcal{V}\}$.

 \begin{rem}
By construction, $\bar Y_v$ and $\bar Y_{v'}$ are disjoint in $\mathcal X$ for any two vertices $v, v'\in\mathcal  V$ with $d_T(v,v')>2$. Endowed with induced length metric,  $\mathcal X$ is a hyperbolic geodesic space but not proper since each $\bar Y_v$ is glued via flat links with infinitely many $\bar Y_{v'}$'s where $d_T(v', v)=2$.
\end{rem}

 


\begin{defn}
Let $g$ be an element in $G$. The translation length of $g$ is defined to be 
$|g| := \inf_{x \in T} d(x, g x)$. 
Let $\mathrm{Axis}(g) = \{ x \in T \vert d(x, g x) = |g| \}$. If $\mathrm{Axis}(g) \neq \emptyset$ and $|g| = 0$ then $g$ is called \emph{elliptic}.  If $\mathrm{Axis}(g) \neq \emptyset$ and  $|g| > 0$, it is called \emph{loxodromic} (or \emph{hyperbolic}).
\end{defn}

\begin{rem}
We note that $\mathrm{Axis}(g) \neq \emptyset$ for any $g \in G$. If $g$ is loxodromic, $\mathrm{Axis}(g)$ is isometric to $\mathbb{R}$, and $g$ acts on $\mathrm{Axis}(g)$ as translation by $|g|$.
\end{rem}

\begin{lem}
\label{lem:index2subgroup}
There exists a subgroup $\dot G$ of index at most 2 in $G$    so that $\dot G$  preserves $\mathcal V_1$ and $\mathcal V_2$ respectively and $G_v\subset \dot G$ for any $v\in T^0$.
\end{lem}
\begin{proof} Observe first that if  $d_T(go, o)=0 \pmod 2$ for some   $o\in T^0$ and $g\in G$, then $d_T(gv, v)=0 \pmod 2$ holds for any $v\in T^0$.
Indeed, if $g$   is elliptic and thus rotates about  a point $o$,   the geodesic $[gv, v]$  for any $v$ is contained in the union $[o,v]\cup[o,gv]$   and thus has even length.    Otherwise, $g$ must be a hyperbolic element and leaves invariant a geodesic $\gamma$ acted upon by translation. By a similar reasoning, if $g$ moves the points on $\gamma$ with even distance, then $d_T(gv, v)=0 \pmod 2$   for any $v\in T^0$.    

Consider now the set $\dot G$ of elements $g\in G$ such that $d_T(gv, v)=0 \pmod 2$ for any $v\in T^0$. Using the tree  $T$ again, if $g,h\in \dot G$, then $d_T(gv, hv)=0 \pmod 2$ for any $v\in T^0$. Thus, $\dot G$ is a group of finite index 2. 
\end{proof}

Let $\dot G$ be the subgroup of $G$ given by the lemma. By Bass-Serre theory, it admits a finite graph  of groups where the underlying graph $\dot{\mathcal G}=T/\dot G$ is bipartite with vertex sets $V=\mathcal V/\dot G$ and $W=\mathcal W/\dot G$ where $\mathcal W:=T^0\setminus \mathcal V$,  and  the vertex groups  are isomorphic to those of $G$.

\begin{lem}\label{ActOnHypspaceLem}
The space $\mathcal X$  is a $\delta$--hyperbolic Hadamard space where $\delta>0$ only depends on the hyperbolicity constants of $\bar Y_v$  ($v\in \mathcal V$). 

If for every $v\in T^0$, $G_v=H_v\times Z(G_v)$  and for each edge $e:=[v,w]\in T^1$, $G_e=Z(G_v)\times Z(G_w)$ then the  subgroup $\dot G <G$  given by Lemma \ref{lem:index2subgroup} acts on  $\mathcal X$    with   the following properties:
\begin{enumerate}
\item
\label{ActOnHypspaceLem:1}
for each $v\in \mathcal V$,  the stabilizer of $\bar Y_v$ is isomorphic to $G_v$ and $H_v$ acts geometrically on $\bar Y_v$, and
\item
\label{ActOnHypspaceLem:2}
    for each $w\in  \mathcal W$, the flat link  $Fl(w)$   admits an isometric group action of $G_w$ so that $G_w$ acts by translation on the line parallel to the binding line and   on the set of flat strips   by the action on the link $Lk(w)$. 
    
    
    
    
    
\end{enumerate}  
\end{lem}

\begin{proof} 
On one hand, $G_e$ acts on the boundary line $\ell_e=F_e\cap \bar Y_v$ through $G_e \rightarrowtail G_e/Z(G_v)=Z(G_w)$.  On the other hand, $Z(G_w)$ acts on the boundary line of the flat strip corresponding to the edge $e$. Since these two actions are compatible with glueing of $Y_v$'s where $v\in Lk(w)$, we can extend the actions on $\bar Y_v$'s and  flat links $Fl(w)$'s to get the desired action of $\dot G$ on $\mathcal X$. 
\end{proof}

\subsection{Q.I. embedding into the product of two hyperbolic spaces}

\label{subsection:qiehyp}
\begin{prop}
\label{prop:qiehyperbolicspaces}
Let $G \curvearrowright X$ be a flip CKA action and $\dot G$ the subgroup in $G$ of index at most $2$ given by Lemma~\ref{lem:index2subgroup}. Let $\mathcal{X}_i$ ($i = 1,2$) be the hyperbolic space constructed in Section~\ref{subsection:flipactions} with respect to $\mathcal{V}_i$. Then there exists a   quasi-isometric embedding map $\phi$ from $X$ to $\mathcal X_1 \times \mathcal X_2$.  

If for every $v\in T^0$, $G_v=H_v\times Z(G_v)$  and for each edge $e:=[v,w]\in T^1$, $G_e=Z(G_v)\times Z(G_w)$ then the above map $\phi$ can be made $\dot G$-equivariant.
\end{prop}

\begin{proof}
Let $\rho \colon X \to T^0$ be the indexed map given by Remark~\ref{rem:indexfunction}. Choose a vertex $v \in T^0$ and a point $x_0 \in Y_v$ such that $\rho(x_0) =v$.
Note that $\rho$ is a $G$--equivariant, hence $\rho(g(x_0)) = g \rho(x_0) = gv$. Without loss of generality, we can assume that $v \in \mathcal{V}_1$.
Let $\tilde X = \dot G (x_0)$ be the orbit of $x_0$ in $X$.

For any $x \in \tilde X = \dot G (x_0)$ then $x = g(x_0)$ for some $g \in \dot G$. We remark here that in general it is possible that $g(x_0)$ belong to several $Y_w$'s for some $w \in T^0$. However, recall that we have the given indexed map $\rho \colon X \to T^0$. This indexed function will tell us exactly which space we should project $g(x_0)$ into, i.e, $g(x_0)$ should project to $\bar Y_{\rho(g(x_0))} = \bar{Y}_{gv}$.

We recall that $\dot G$ has finite index in $G$ and it preserves $\mathcal V_1$ and $\mathcal V_2$.

{\bf Step 1 : Construct a quasi-isometric embedding map $\phi \colon \tilde X \to \mathcal{X}_1 \times \mathcal{X}_2$}

We are going to define the map $\phi = \phi_1 \times \phi_2: \tilde  X\to \mathcal X_1\times \mathcal X_2$ where $\phi_i: \tilde X\to \mathcal X_i$.  We first define a map $\phi_1: \tilde X\to \mathcal X_1$. 

For any $x \in \tilde X$ then $x = g(x_0)$ for some $g \in \dot G$, and thus $g(x_0) \in Y_{gv}$. Since we assume that $v \in \mathcal V_1$ and $\dot G$ preserves $\mathcal V_1$, it follows that $gv \in \mathcal V_1$. We define  $\phi_1 (x) : = \pi_{\bar{Y}_{gv}}(x)$  where $\pi_{\bar Y_{gv}}$ is the projection of $Y_{gv} = \bar{Y}_{gv} \times \R$ to the factor $\bar Y_v$.  We define $\phi_2 (x)$ to be the point on the binding line of the flat link $Fl(v)$ so that its $\mathbb R$--coordinate is the same as that of   $x$ in the metric product $Y_{gv} = \bar Y_{gv} \times \R$.

{\bf Step 2: Verifying $\phi$ is a quasi-isometric embedding.}

We are now going to show that $\phi = \phi_1 \times \phi_2  \colon \tilde{X} \to \mathcal{X}_1 \times \mathcal{X}_2$ is a quasi-isometric embedding. Before   getting into the proof, we clarify here an observation that will be used later on.

Observation: By the tree-like construction of $\mathcal X$, any geodesic $\alpha$ in $\mathcal X$ with endpoints $\alpha_-\in \bar Y_{v_1}, \alpha_+\in \bar Y_{v_2}$   crosses $\bar Y_v$   for  alternating vertices $v\in [v_1, v_2]$ in their order appearing in the interval, where $[v_1, v_2]$ is the geodesic in the tree $T$.  Using the convexity of  boundary lines in a hyperbolic CAT(0) space, we see that  the intersection $\alpha\cap \bar Y_v$ connects two boundary lines $\ell, \ell'$ in $\bar Y_v$ so that the projection $\pi_\ell(\ell')$ is uniformly close to $\alpha$. We can then construct a quasi-geodesic $\beta$  in $\mathcal{X}$  with the same endpoints as $\alpha$ so that $\beta\cap \bar Y_v$ connects $\pi_\ell(\ell')$ to $\pi_{\ell'}(\ell)$.

{\bf Claim:} There exists a constant $C \ge 2$ such that $d(x,y) /C -C \le d(\phi(x) , \phi(y)) \le C d(x,y) +D$ for any $x, y \in \tilde X$.

Indeed, let $g \in \dot G$ and $g' \in \dot G$ be two elements such that $x = g(x_0)$ and $y = g'(x_0)$. Note that $x \in Y_{gv}$ and $y \in Y_{g'v}$. We consider the following cases.

Case~1: $gv = g'v$. In this case, it is easy to see the claim holds since $$Len^1[x,y]=Len^1(\phi_1([x, y]))+Len^1(\phi_2([x, y]))$$

Case~2: $g v \neq g' v$.  Note that they both belong to $\mathcal V_1$, so $d_{T}(gv, g'v)$ is even. Let $2n = d_{T}(gv, g'v)$.
We write 
\[
[gv, g'v ] = e_1 \cdots e_{2n}
\] as the edge paths in $T$. Let $v_{i-1}$ be the initial vertex of $e_{i}$ with $i = 1, \dots , 2n$ and $v_{2n}$ be the terminal vertex of $e_{2n}$.  {We note that $v_0, v_2, \dots, v_{2n} \in \mathcal V_1 $ and $v_1, v_3, \dots, v_{2n-1} \in \mathcal V_2$.
}

With notations in in Definition \ref{SpecialPathDefn}, the special path $\gamma$  between $x, y$  decomposes  as the concatenation of geodesics:
$$
\gamma=[x, p_1][p_1,p_2]\cdots [p_{2n}, y].
$$ 

Denote $(x_1, x_2) = (\phi_1(x), \phi_2(x)) \in \mathcal X_1\times \mathcal X_2$ and  $(y_1, y_2) = (\phi_1(y), \phi_2(y)) \in \mathcal X_1\times \mathcal X_2 $. By  the above observation, we connect $x_1$ to $y_1$   by a quasi-geodesic $\beta_1$ in $\mathcal X_1$  so that  whenever $\beta_1$ passes through $\bar{Y}_{v_2} , \, \bar{Y}_{v_4}, \, \cdots, \bar{Y}_{v_{2n-2}}$, it is orthogonal to the boundary lines. 
In this way, we can write $\beta_1$ as the concatenation of geodesic segments $\beta_1^0,\beta_1^1, \beta_1^2, \cdots, \beta_1^{2n}$, where     $\beta_1^{2i +1}$  are maximal segments contained   in the flat links. The first   $\beta_1^0$ and last $\beta_1^{2n}$ may have overlap with boundary lines, and the other $\beta_1^{2i}$   are orthogonal to the boundary lines of $\bar Y_{v_{2i}}$.  

Similarly, let $\beta_2$ be a quasi-geodesic from $x_2$ to $y_2$   in $\mathcal X_2$  as the concatenation of geodesic segments $\beta_2^0,\beta_2^1, \beta_2^2, \cdots, \beta_2^{2n}$. 


We relabel $x$ by $p_0$ and relabel $y$ by $p_{2n+1}$. 
For each vertex $v_i$, let $\pi_{\bar Y_{i}}$ and $\pi_{\R_{i}}$ denote the projections of $Y_{v_i} = \bar Y_{v_i} \times \R$ to the factor $\bar Y_i$ and $\R$ respectively. 

By the construction of $\phi_1$ and $\phi_2$ we note that there exists a constant $A = A(\mathcal{X})$ such that
\[
  Len_{\mathcal{X}_1} (\beta_1 ) \sim_{A} \,\,\, \sum_{i=0}^{n} Len_{X} \bigl ( \pi_{\bar{Y}_{2i}} [ p_{2i}, p_{2i+1}] \bigr ) + \sum_{i=0}^{n-1} Len_{X} \bigl ( \pi_{\R_{2i+1}} [ p_{2i+1}, p_{2i+2} ] \bigr )
\]
and
\[
  Len_{\mathcal{X}_2} (\beta_2 ) \sim_{A} \,\,\, \sum_{i=0}^{n} Len_{X} \bigl ( \pi_{\R_{2i}} [ p_{2i}, p_{2i+1}] \bigr ) + \sum_{i=0}^{n-1} Len_{X} \bigl ( \pi_{\bar{Y}_{2i+1}} [ p_{2i+1}, p_{2i+2} ] \bigr )
\]
Summing over two equations above, we obtain
\begin{equation}
\label{equation4:prop3.7}
   d(x,y) \sim_{B} \bigl ( Len_{\mathcal{X}_1}(\beta_1) + Len_{\mathcal{X}_2} (\beta_2) \bigr ) 
\end{equation}
for some constant $B = B (A)$.

Since $\beta_i$ is a $(\kappa, \kappa)$--quasi-geodesic connecting two points $\phi_i(x)$ and $\phi_i(y)$ (for some uniform constant $\kappa$ that does not depend on $x$, $y$), we have that $\ell(\beta_i) \sim_{\kappa} d(x_i, y_i)$ with $i = 1,2$. This fact together with formula (\ref{equation4:prop3.7}) and the fact $d(\phi(x), \phi(y)) \asymp_{\sqrt{2}} d(x_1, y_1) + d(x_2, y_2)$ give a constant $c = C(B, \kappa)$ such that $d(x, y) \sim_{C} d(\phi(x), \phi(y))$. The claim is verified.  Therefore, $\phi$ is a quasi-isometric embedding.
\end{proof}

\subsection{Q.I. embedding into a finite product of trees}  
\label{subsection:qieproducttrees}
In this section, we first recall briefly the work of Bestvina-Bromberg-Fujiwara  \cite{BBF} on constructing  a quasi-tree of spaces. The reminder is then to produce a collection of quasi-lines to establish a distance formula for   geometric actions of hyperbolic groups, see Lemma \ref{YDistFormulaLem}. The result is not new (cf. \cite[Prop. 3.3]{BBF2}), but the construction of quasi-lines is new and generalizes to certain non-proper actions, see Lemma \ref{ConeofYDistFormulaLem}.

We shall make use of the work of . Their theory  applies to any collection of spaces $\mathbb Y$   equipped with a family of  {\it projection maps}  $$\{\pi_{Y}: \mathbb Y\setminus \{Y\}\times \mathbb Y\setminus \{Y\}\to Y\}_{Y\in \mathbb Y}$$ satisfying the so-called \textit{projection axioms} with  projection constant $\xi\ge 0$. The precise formulation of projection axioms is irrelevant here.  We only mention that their results apply  to a collection of quasi-lines $\mathbb A$ with bounded projection property in a (not necessarily proper) hyperbolic space $Y$, where the projection maps $\pi_\gamma$ for $\gamma\in \mathbb A$ are shortest point projections to $\gamma$ in $Y$.   Then $(\mathbb A, {\pi_\gamma })$ satisfies projection axioms with projection constant $\xi$ (see \cite[Proposition~2.4]{BBF2}).


 



Fix $K>0$. Following \cite{BBF}, a quasi-tree of spaces $\mathcal C_K(\mathbb A) $ is constructed with a underlying quasi-tree (graph) structure where every vertex represents a quasi-line in $\mathbb A$ and two quasi-lines $\gamma, \gamma'$ are connected by an edge of length 1 from $\pi_\gamma(\gamma')$ to $\pi_{\gamma'}(\gamma)$. If $K>4\xi$, then $\mathcal C_K(\mathbb A)$ is a unbounded quasi-tree. 

If $\mathbb A$ admits a group action of $G$ so that $\pi_{g\gamma}=g\pi_\gamma$ for any $g\in G$ and $Y\in \mathbb A$, then $G$ acts by isometry on $\mathcal C_K(\mathbb A)$.

By \cite[Lemma 4.2, Corollary 4.10]{BBF}, every quasi-line  $\gamma\in \mathbb A$ with induced metric from $Y$ is totally geodesically embedded in $\mathcal C_K(\mathbb A)$ and the shortest projection maps from $\gamma'$ to $\gamma$ in the quasi-tree $\mathcal C_K(\mathbb A)$  coincides with the projection maps $\pi_\gamma(\gamma')$ up to uniform finite Hausdorff distance.

By abuse of language, for both $x, y\in Y$ and $x, y\in \mathcal C_K(\mathbb A)$, we denote $$d_\gamma(x, y)=diam(\pi_\gamma(\{x,y\})$$ where the projections in the right-hand are understood in $Y$ and   $ \mathcal C_K(\mathbb A)$ accordingly.  The above discussion implies that the two projections gives the same value up to a uniform bounded error.   

Set $[t]_K=t$   if $t\ge K$ otherwise $[t]_K=0$. We now summarize what we need from \cite{BBF, BBF2} in the present paper. 
\begin{prop}  
 \label{BBFDistanceThm}
Let $\mathbb A$ be a collection of   quasi-lines    in a $\delta$--hyperbolic space $Y$. If there is $\theta>0$ such  that $diam (\pi_\beta (\alpha)) \le \theta$ for all $\alpha\ne \beta\in \mathbb A$, then $(\mathbb A, {\pi_\gamma })$ satisfies the projection axioms with projection constant $\xi$ depending only on $\theta$. Moreover, for any $x, y\in \mathcal C_K(\mathbb A)$,
$$
d_{\mathcal C_K(\mathbb A)}(x,y) \sim_K    \sum_{\gamma\in \mathbb A} [d_\gamma (x, y)]_K 
$$
for all $K \ge 4\xi$. 
\end{prop}
\begin{rem} 
In \cite[Proposition~2.4]{BBF2},  the above formula is stated  for $(\mathbb A, {\pi_\gamma })$ with the \textbf{strong} projection axioms. However, by \cite[Theorem 2.2]{BBF2}, a modification $\pi_\gamma'$ of projection maps $\pi_\gamma$ within finite Hausdorff distance can always be done so that  $(\mathbb A, {\pi_\gamma'})$ satisfies the strong projection axioms. Thus, the same formula still holds with original projection maps.  
\end{rem}
  
As a corollary, the distance formula still works when the points  $x, y$ are perturbed up to bounded error. 
  
\begin{cor}\label{BBFDistanceCor}
Under the assumption of Theorem \ref{BBFDistanceThm}, if $d(x, x'), d(y, y')\le R$ for some $R>0$, then exists $K_0=K_0(R,\xi, \delta)$ such that 
$$
d_{\mathcal C_K(\mathbb A)}(x,y) \sim_K \sum_{\gamma\in \mathbb A} [d_\gamma (x', y')]_K 
$$
for all $K \ge 2K_0$.
\end{cor}
\begin{proof}
If $d(x, x'), d(y, y')\le R$ then there exists  a constant $K_0=K_0(R, \xi, \delta)$ such that  $|d_\gamma (x, y)-d_\gamma(x',y')|\le K_0$ for any $\gamma\in \mathbb A$.   Assuming $d_\gamma (x, y)>K\ge 2K_0$  then $d_\gamma (x',y')\ge K_0$ we see that $${1\over 2}[d_\gamma (x',y')]_K \le [d_\gamma (x, y)]_K\le 2 [d_\gamma (x',y')]_K$$
yielding the desired formula.
\end{proof}
  
\begin{defn}[Acylindrical action]\cite{Bow08}\cite{Osin}
Let $G$ be a group acting by isometries on a metric space $(X,d)$. The action of $G$ on $X$ is called {\it acylindrical} if for any $r \ge 0$, there exist constants $R, N \ge 0$ such that for any pair $a, b \in X$ with $d(a, b) \ge R$ then we have 
\[
\# \bigl \{  g \in G \,|\, d(ga, a) \le r \,\, \textup{and} \,\, d(gb, b) \le r \bigr \} \le N
\]
\end{defn}

By \cite{Bow08}, any nontrivial isometry of acylindrical group action on a hyperbolic space is either elliptic  or loxodromic.  A $(\lambda, c)$-quasi-geodesic $\gamma$ for some $\lambda, c>0$ is referred to as a \textit{quasi-axis} for a loxodromic element $g$, if $\gamma, g\gamma$ have (uniform) finite Hausdorff distance.

The following property in hyperbolic groups is probably known to experts, but  is referred to a more general result \cite[Lemma 2.14]{YANG}   since   we could not locate a precise statement as follows. A group is called \textit{non-elementary} if it is neither finite nor virtually cyclic.

\begin{lem}\label{ExtensionLem}
Let $H$ be a non-elementary   group admitting a co-bounded and acylindrical action on  a $\delta$--hyperbolic space $(\bar Y, d)$. Fix a basepoint $o$. Then there exist a set $F\subset H$ of three loxodromic elements and $\lambda, c>0$  with the following property. 

For any $h\in H$ there exists $f\in F$ so that $hf$ is a loxodromic element and the   bi-infinite path $$\gamma=\cup_{i\in \mathbb Z}(hf)^{i}([o, ho][ho, hfo])$$  is a $(\lambda, c)$--quasi-geodesic.
\end{lem} 
\begin{proof}[Sketch of the proof]
This follows from the  result \cite[Lemma 2.14]{YANG} which  applies to any isometric action of $H$ on a metric space with a set $F$ of three pairwise independent contracting elements (loxodromic elements in hyperbolic spaces). If   $\mathbb X$ denotes the set of $G$--translated quasi-axis of all elements in $F$, the pairwise independence  condition is equivalent (defined) to be the  bounded projection property of $\mathbb X$. Thus, the existence of such $F$ is clear in a proper action of a non-elementary group. For acylindrical actions, this is also well-known, see \cite[Proposition 3.4]{BBF2} recalled  below.   
\end{proof}

\begin{prop}\cite{BBF2}
\label{FiniteDblCosetsProp}
Assume that a hyperbolic group $H$ acts acylindrically on a hyperbolic space $\bar Y$. For a loxodromic element $g\in H$, consider    the set $\mathbb A$ of all $H$-translates of a fixed $(\lambda, c)$-quasi-axis of $g$ for given $\lambda, c>0$.  Then there exist 
$\theta =\theta(\lambda, c)> 0$ and $N =N(\lambda, c)> 0$ such that for any $\gamma\in \mathbb A$, the set $$\{h\in G: diam(\pi_\gamma(h\gamma))\ge  \theta\}$$ consists of at most $N$     double $E(g)$-cosets.

In particular, there are at most $N$ distinct pairs $(\gamma, \gamma')\in  \mathbb A\times   \mathbb A$  satisfying  $diam(\pi_\gamma(\gamma'))> \theta$ up to the action of $H$.
\end{prop}
The following corollary can be derived from the ``in particular" statement using hyperbolicity. 
\begin{cor}\label{LocSegmFiniteCor}
Under the assumption of Proposition \ref{FiniteDblCosetsProp}, for any $R>0$, there exist  constants $\theta=\theta(\lambda, c, R), N=N(\lambda, c, R)>0$ such that for any geodesic segment $p$ of length $\theta$, we have
$$
\sharp \{\gamma \in \mathbb A: p\subset N_R(\gamma)\}\le N.
$$
\end{cor}

\begin{conv}\label{ConvQuasiLine}
When speaking of quasi-lines in hyperbolic spaces with actions satisfying Lemma \ref{ExtensionLem},   we always mean  $(\lambda, c)$--quasi-geodesics where  $\lambda, c>0$  depend on $F$ and $\delta$.
\end{conv}

\begin{lem} \label{YDistFormulaLem}
Let $H$ be a non-elementary   group admitting a proper and cocompact action on  a $\delta$--hyperbolic space $(\bar Y, d)$. Assume that $\mathbb L$ is a $H$--finite collection of quasi-lines.  Then for any sufficiently large $K>0$,  there exist  a  $H$--finite collection of quasi-lines $\mathbb L\subset \mathbb A$    in $\bar Y$  and a constant $N=N(K, \delta, \mathbb A)>0$, such that  for any $x, y\in \bar Y$, the following holds
$$
\frac{1}{N} \sum_{\gamma\in \mathbb A} 
[d_\gamma(x, y)]_K\le d(x, y) \le   2\sum_{\gamma\in \mathbb A } 
[d_\gamma(x, y)]_K + 2K. 
$$
\end{lem}
\begin{rem}
The statement of Lemma~\ref{YDistFormulaLem}  (with torsion allowed here) is a re-package of Proposition 3.3 and Theorem 3.5 in \cite{BBF2}. Our proof is different and generalizes to certain co-bounded and acylindrical actions coming from relatively hyperbolic groups on their relative Cayley graphs. See  Lemma \ref{ConeofYDistFormulaLem} for details.  
\end{rem}
\begin{proof}   
Fixing a point $o\in \bar Y$,  the co-bounded action of $H$   on $(\bar{Y}, d)$     gives   a constant $R>0$ such that $N_R(Ho)=\bar Y$. By hyperbolicity, if $\gamma$ is a $(\lambda, c)$--quasi-geodesic, then there exists a constant $C=C(\lambda, c, R)>0$ such that   $diam([x,y]\cap N_R(\gamma))>C$ implies 
\begin{equation}\label{ProjBdEQ}
|d_\gamma(x, y)-diam([x,y]\cap N_R(\gamma))|\le C.
\end{equation}

Let $\theta=\theta(R), N_0=N_0(R)$ be the constant given by Corollary \ref{LocSegmFiniteCor} for the geometric (so acylindrical)  action of $H$ on $\bar Y$. 
 
Fix {$K> 2C+2\theta$} and denote $\tilde K=K+C$.  Since the action is proper, the set  $S=\{h\in H: |d(o, ho)- \tilde K|\le 2R\}$  is finite.  Let us consider the set $\tilde S$ of loxodromic elements $hf$ where $h\in S$ and $f\in F$ is provided by Lemma \ref{ExtensionLem}. Note that $\sharp \tilde S=\sharp S$.  Let $\mathbb A$ be the set of all $H$--translated axis of $hf\in \tilde S$. It is possible that $\sharp \mathbb A/H\le \sharp \tilde S$ since two elements in $\tilde S$ may be conjugate.  

Assume that $d(x,y)> \tilde K$. Consider a geodesic $\alpha$ from $x$ to $y$ and choose points $x_i$ on $\alpha$ for $0\le i\le n+1$ such that $d(x_i, x_{i+1})= \tilde K$ for $0\le i \le n-1$ and $d(x_{n},x_{n+1})\le  \tilde K$ where $x_0=x, x_{n+1}=y$. Since $N_R(Ho)=\bar Y$,  there exists $h_i\in H$ so that $d(x_i, h_io)\le R$.  It implies that $\tilde K- 2R \le d(o, h_{i}^{-1}h_{i+1}o) \le  \tilde K+2R$, and thus we have $h_i^{-1}h_{i+1}\in S$ for $0\le i\le n-1$. Noting that $[h_io,h_{i+1}o]$ is contained in a $H$-translated axis of some loxodromic element  in $\tilde S$, we thus obtain $n$ axis  $\gamma_0, \cdots, \gamma_{n-1} \in \mathbb A$ (with possible multiplicities: $\gamma_i=\gamma_j$ for $i\ne j$)  satisfying $diam(N_R(\gamma_i)\cap \alpha)\ge \tilde K$  so that $$\alpha\setminus [x_n, x_{n+1}]\subset  \bigcup_{0\le i\le n-1} N_R(\gamma_i)\cap \alpha$$ which yields
$$ 
Len(\alpha) \le \sum_{0\le i\le n-1} diam(N_R(\gamma_i)\cap \alpha)+ \tilde K    
$$ 
where the constant $\tilde K$ bounds the length of the last segment $[x_{n}, x_{n+1}]$.  

By Equation (\ref{ProjBdEQ}),   $d_{\gamma_i}(x,y)\ge diam(N_R(\gamma_i)\cap \alpha)-C \ge K>2C$ and then $diam(N_{R}(\gamma_i) \cap \alpha) \le d_{\gamma_i}(x, y) + C \le 2d_{\gamma_i}(x,y)$. Thus,  we obtain
$$
Len(\alpha) \le  \sum_{0\le i\le n-1}  2 [d_{\gamma_i}(x,y)]_K +\tilde K,
$$
which implies the upper bound over $\gamma \in \mathbb A$. Of course, the upper bound holds as well  after adjoining $\mathbb L$ into $\mathbb A$.

The remainder of the proof is to prove the lower bound. Let $\mathbb B$   be the set  of quasi-lines $\gamma \in \mathbb A\cup \mathbb L$ satisfying   $d_{\gamma}(x,y) > K=2\theta+2C$. Note that the set of axis $\gamma_0,\cdots,  \gamma_{n-1}$ obtained as above is included  into $\mathbb B$.

By the proper action of $H$ on $\bar Y$, the $H$--finite $\mathbb L$ has bounded intersection so does  $\mathbb L\cup \mathbb A$. Thus, there is  $D=D(\mathbb A, \mathbb L, \tilde K, R)>0$ so that $diam(N_R(\gamma)\cap N_R(\gamma'))< D$ for any $\gamma\ne \gamma'\in \mathbb A \cup \mathbb L$.   In particular, different $N_{R}(\gamma) \cap \alpha$'s have overlaps bounded above by $D$.


By Eq. (\ref{ProjBdEQ}), we obtain  $diam(N_{R}(\gamma) \cap \alpha) \ge d_{\gamma}(x, y) - C \ge {1\over 2} d_{\gamma}(x,y) \ge \theta$.  As mentioned-above in the second paragraph, by Corollary \ref{LocSegmFiniteCor},  any segment of length $\theta$ is covered at most $N_0$ times by  the   $R$-neighborhood of  quasi-lines in $\mathbb B$.  Thus, there exists a constant $N=N(D, N_0)>0$ such that
$$
N\cdot Len(\alpha) \ge \sum_{\gamma\in \mathbb B} diam(N_R(\gamma)\cap \alpha) \ge {1\over 2}\sum_{\gamma\in \mathbb B} [d_{\gamma}(x,y)]_K.   
$$
The proof is completed by renaming  $\mathbb A:= \mathbb A\cup  \mathbb L.$
\end{proof}

\subsection{Q.I. embedding into a finite product of trees}  
\label{subsection:qieproducttrees}
This subsection is devoted to the proof of Theorem \ref{thm:main3}. The results obtained here are not used in other places, and so can be skipped if the reader is interested in the stronger conclusion, the property (QT), under stronger assumption. 

We start by explaining the choice of the constants and the collection of quasi-lines $\mathbb A$ in $\mathcal{X}_1$ that will be used in the rest of this subsection. 

{\bf The constants $D$ and $\theta$ and $\xi = \xi(\theta)$:} Let $\mathcal{X}_1$ and $\mathcal{X}_2$ be two $\delta$--hyperbolic spaces given by Lemma~\ref{ActOnHypspaceLem} where $\delta>0$ depends on the hyperbolicity constants of $\bar Y_v$.

Note that each $\bar Y_v$ for $v\in \mathcal V_1$ is isometrically embedded into $\mathcal X_1$  and thus $\delta$--hyperbolic. We follow the  Convention \ref{ConvQuasiLine} on the quasi-lines which are $(\lambda, c)$--quasi-geodesics in $\bar Y_v$ and $\mathcal X_1$. 

By the $\delta$--hyperbolicity of $\mathcal X_1$, there exist constants $D, \theta>0$ depending on $\delta$   (and also $\lambda, c$) such that if any ($(\lambda,c)$--)quasi-lines  $\alpha\ne \beta$ have a distance at least $D$ then $diam (\pi_\beta (\alpha)) \le \theta$.

We then obtain the projection constant $\xi=\xi(\theta)$   by Proposition \ref{BBFDistanceThm}. 

{\bf The collection of quasi-lines $\mathbb A$ in $\mathcal{X}_1$:}
Fix any sufficiently large number  $K>\max\{4\xi, \theta, 2\}$ depending on $\mathbb L_v$, where $\mathbb L_v$ is the collection of boundary lines of $\bar Y_v$.  By Lemma \ref{YDistFormulaLem}, there exist  a locally finite collection of quasi-lines $\mathbb L_v\subset \mathbb A_v$    in $\bar Y_v$ and a constant $N=N(K, \mathbb A_v, \delta)>0$ such that 
\begin{equation}\label{YvDistFormulaEQ}
\frac{1}{N}\sum_{\gamma\in \mathbb A_v } 
[d_\gamma(x, y)]_K\le d_{\bar Y_v} (x, y) \le   2\sum_{\gamma\in \mathbb A_v } 
[d_\gamma(x, y)]_K + 2K
\end{equation}
for any $x, y\in \bar Y_v$.
 Since there are only finitely many $\dot G$--orbits of $(H_v, \bar Y_v)$ we assume furthermore $\mathbb A_w=g\mathbb A_v$ if $w=gv$  for $g\in \dot G$. Then $$\mathbb A := \cup_{v\in\mathcal  V_1} \mathbb A_v$$ is a locally finite collection of quasi-lines in $\mathcal X_1$, preserved by the group $\dot G$. 
 
We use the following lemma in the proof of Proposition~\ref{prop:distanceformulaX_1} that gives us a distance formula for $\mathcal{X}_1$.

\begin{lem}
\label{lem:Claim1}
There exists a  constant $L>0$ depending only on $K$ with the following properties. 
\begin{enumerate}
\item
For any $\ell\ne \gamma\in \mathbb A$, we have $diam(\pi_\ell(\gamma))\le L$.
\item
For any $v\in \mathcal V_1$ and     $x, y\in \bar Y_v$,  there are at most $L$ quasi-lines $\gamma$ in $\mathbb A \setminus  \mathbb A_v$  such that  $L\ge d_\gamma(x, y)\ge K$.
\end{enumerate}
\end{lem}

\begin{proof}
Since there are only finitely many $\bar Y_w$'s up to isometry, and $\mathbb A_w=g\mathbb A_v$ if $w=gv$,  the union    $\mathbb A$ of quasi-lines containing $\cup_{w\in\mathcal  V_1} \mathbb L_w$ is uniformly locally finite: any ball of a fixed radius  in $\mathcal X_1$ intersects a uniform number of quasi-lines depending only on the radius.  By the hyperbolicity of $\mathcal X_1$, the local finiteness implies the bounded projection property, so  gives the desired constant $L$ in the assertion (1).  

By the construction of $\mathcal X_1$,  the shortest projection of a point $x\in \bar Y_v$ to  $\gamma \in \mathbb A_w$ for $w\ne v$ has to pass through a boundary line $\ell\in \mathbb L_w$  of $\bar Y_w$, so is contained in the projection of  $\ell$ to $\gamma$.  By the assertion (1) we have   $d_\gamma(x, y)\le diam(\pi_\gamma(\ell)) \le L$.   If   $d_\gamma(x, y)\ge K\ge \theta $ for $x, y\in \bar Y_v$ and $\gamma\in \mathbb A_w$ with $w\ne v$, then  $d(\gamma, \ell)\le D$ by the above defining property of $D$ and $\theta$. By local finiteness, there are at most $L=L(D)$ quasi-lines $\ell$ with this property, proving the assertion~(2).   
\end{proof}

\begin{prop}[Distance formula for $\mathcal{X}_1$]
\label{prop:distanceformulaX_1}
For any $x, y\in \mathcal X_1$, there exists a  constant $\mu=\mu(L, K)>0$ such that 
\begin{equation}\label{DistFormulaEQ}
\begin{array}{cc}
{1\over \mu}\sum_{\gamma\in \mathbb A } 
[d_\gamma(x, y)]_{K}+    d_T(\rho(x), \rho(y)) - L^2 \\
\\
\le d_{\mathcal X_1} (x, y) \le \\
\\
   \mu \sum_{\gamma\in \mathbb A } 

[d_\gamma(x, y)]_{K} + 4K   \cdot d_T(\rho(x), \rho(y)).
\end{array}
\end{equation}
\end{prop}

\begin{proof}
Since the $1$--neighborhood of the union $\cup_{v\in\mathcal V_1} \bar Y_v$ is $\mathcal X_1$,  assume for simplicity $x\in \bar Y_{v_1}$ and $y\in \bar Y_{v_n}$ where $v_1=\rho(x), v_n=\rho(y)\in \mathcal V_1$.  By the construction of $\mathcal X_1$, a geodesic $[x,y]$ travels through $\bar Y_{v_i}$ and then flat links $Fl(w_i)$, where $\{v_i\in \mathcal V_1: 1\le i\le n\}$ and $\{w_i\in \mathcal V_2: 1\le i\le n-1\}$ appear alternatively on $[v_1, v_n]\subset T$. Thus, $d_T(v_1, v_n)=2n-2$.

Let us denote the exit point on the boundary line $\ell_i$ of $\bar Y_{v_i}$ and entry point on  $\ell_{i+1}'$ of $\bar Y_{v_{i+1}}$ by $y_i$ and $x_{i+1}$ respectively for $1\le i \le n$  where $x_1:=x$ and  $y_n : = y$ by convention.  
Thus,
\begin{equation}\label{GeodFormulaEQ}
d_{\mathcal X_1}(x, y) - \sum_{1\le i\le n-1} d_{\mathcal X_1}(y_i, x_{i+1}) = \sum_{1\le i\le n} d_{\bar Y_{v_i}}(x_i, y_i).
\end{equation}

Therefore, we shall derive  (\ref{DistFormulaEQ}) from (\ref{GeodFormulaEQ}) which requires to apply  the formula  (\ref{YvDistFormulaEQ}) for  $d_{\bar Y_{v_i}}  (x_i, y_{i})$. To that end, we need the following estimates. Recall that    $  \asymp_{L,K} $ means the equality holds up to a multiplicative constant depending on $ L, K$.
\begin{claim}
\label{claim:1}
\begin{enumerate}
\item
If there is   $\gamma\in \mathbb A_{v_i}$ such that $d_\gamma(x_i, y_{i})\ge K$ then
\begin{equation}\label{SameProjEQ}
[d_\gamma(x_i, y_{i})]_K  \asymp_{L,K} [d_\gamma(x, y)]_{ K}
\end{equation}

\item
$d_{\mathcal X_1}(y_i, x_{i+1})\ge 2$. If  $d_{\mathcal X_1}(y_i, x_{i+1})>K+2$, then $$[d_{\ell_i}(y_i, x_{i+1})]_K \asymp_{L,K} [d_{\ell_i}(x, y)]_K, \;\;d_{\ell_{i+1}'}(y_i, x_{i+1})\asymp_{L,K} [d_{\ell_{i+1}'}(x, y)]_K$$
\end{enumerate}
\end{claim}
\begin{proof}[Proof of the Claim~\ref{claim:1}]
If there is   $\gamma\in \mathbb A_{v_i}$ such that $d_\gamma(x_i, y_{i})\ge K$   we then have $$|d_\gamma(x, y)-d_\gamma(x_i, y_{i})|\le diam(\pi_\gamma(\ell_i))+diam(\pi_\gamma(\ell_{i}'))\le 2L,$$
where Lemma~\ref{lem:Claim1} is applied,  and  after taking the cutoff function $[\cdot]_K$, 
$$
|[d_\gamma(x, y)]_{K} - [d_\gamma(x_i, y_{i})]_K|\le 2L+K.
$$
This in turn  implies (\ref{SameProjEQ}).

Recall that $[y_i, x_{i+1}]$ is contained in the union of two flat strips with width $1$ in a flat link, and is from one boundary line $\ell_i$ to the other $\ell_{i+1}'$.  Thus, $ d_{\mathcal X_1}(y_i, x_{i+1})\ge 2$.    If  $d_{\mathcal X_1}(y_i, x_{i+1})>K+2$ is assumed, then $d_{\ell_i}(y_i, x_{i+1})>K$ and $d_{\ell_{i+1}'}(y_i, x_{i+1})>K$. The assertion (2) follows similarly as above.
\end{proof}

Recalling $K\ge 2$, the assertion (2) of the Claim~\ref{claim:1} implies a constant $\mu_1=\mu_1(L, K)>1$ such that
\begin{equation}\label{SumFlatlinksEQ}
d_T(\rho(x), \rho(y)) \le \sum_{1\le i\le n-1} d_{\mathcal X_1}(y_i, x_{i+1})  \le \mu_1 \sum_{\ell\in \mathbb A} [d_{\ell}(x, y)]_K+2K d_T(\rho(x), \rho(y)).
\end{equation}

Using (\ref{SameProjEQ}), we now replace   $[d_\gamma(x_i, y_{i})]_K$  by $[d_\gamma(x, y)]_{K}$ in the formula  (\ref{YvDistFormulaEQ}) for $d_{\bar Y_{v_i}}  (x_i, y_{i})$. 
Hence, there exists      a constant $\mu_2=\mu_2(K, L)>1$  so that 
$$
\frac{1}{\mu_2}\sum_{\gamma\in \mathbb A_{v_i} } 
[d_\gamma(x, y)]_{K}\le d_{\bar Y_{v_i}}  (x_i, y_{i}) \le    \mu_2 \sum_{\gamma\in \mathbb A_{v_i} } 
[d_\gamma(x, y)]_{K} +2K.
$$

Noting $d_T(\rho(x), \rho(y))=2n-2$, we deduce from Eq. (\ref{GeodFormulaEQ}) and (\ref{SumFlatlinksEQ}) that 
$$
d_{\mathcal X_1} (x, y) \le   (\mu_1+\mu_2) \sum_{1\le i \le n} \sum_{\gamma\in \mathbb A_{v_i} } 
[d_\gamma(x, y)]_{K}  +  4K  \cdot d_T(\rho(x), \rho(y)) $$  
so the upper bound in (\ref{DistFormulaEQ})  follows by setting $\mu:=\mu_1+\mu_2$. 

We now derive the lower bound from those of Eq. (\ref{GeodFormulaEQ}) and (\ref{SumFlatlinksEQ}):
$$
\begin{array}{rl}
    d_{\mathcal X_1} (x, y) & \ge \sum_{1\le i \le n} d_{\bar Y_{v_i}}  (x_i, y_{i}) + d_T(\rho(x), \rho(y))  \\
     & \ge  \frac{1}{\mu_2} \sum_{1\le i \le n} \sum_{\gamma\in \mathbb A_{v_i} } [d_\gamma(x, y)]_{ K}  +       d_T(\rho(x), \rho(y)) 
\end{array}
$$  
By the Claim~\ref{claim:1},  there are at most $L$   quasi-lines $\gamma \in \cup \{\mathbb A_{v}:v\in \mathcal V_1\setminus [\rho(x), \rho(y)]^0\}$ satisfying $L\ge [d_\gamma(x, y)]_{K}>0$. Hence, the following holds    
$$
\begin{array}{rl}
    d_{\mathcal X_1} (x, y) & \ge    \frac{1}{\mu_2} \sum_{v\in \mathcal V_1} \sum_{\gamma\in \mathbb A_{v} } [d_\gamma(x, y)]_{ K}  +       d_T(\rho(x), \rho(y))-L^2 
\end{array}
$$
completing the proof of the lower bound.
\end{proof}

\begin{lem}
\label{lem:partitionofA}
The collection 
$\mathbb A$ can be written as a union (possibly non-disjoint) $\mathbb A_1\cup  \cdots \cup \mathbb A_n$ with the following properties for each $\mathbb A_i$:
\begin{enumerate}
    \item 
    for any two quasi-lines $\alpha\ne \beta\in \mathbb A_i$ we have $d(\alpha, \beta) \ge D$,
    \item
     the $(D+R)$--neighborhood of the union $\cup_{\gamma\in \mathbb A_i} \gamma$ contains $\mathcal X_1$,
     \item  for any $K > 4\xi$   the quasi-tree of quasi-lines $(\mathcal C_K(\mathbb A_i), d_{\mathcal C_i})$ is a quasi-tree.
\end{enumerate}
\end{lem}
\begin{proof}
Since $H_v$ acts geometrically on $\bar Y_v$ for $v\in \mathcal V_1$ and $\mathcal V_1$ is $\dot G$--finite, there exists a constant $R>0$ such that       the $R$--neighborhood of the union   $\cup_{\gamma\in \mathbb A_v} \gamma$ contains $\bar Y_v$ for each $v\in \mathcal V_1$. Since $\mathbb A$ is locally finite and $\dot G$--invariant, the $D$--neighborhood of any quasi-line in $\mathbb A$ intersects  $n$ quasi-lines from $\mathbb A$ for some $n=n(D)\ge 1$.  
 
We  can now  write $\mathbb A$ as the (possibly non-disjoint) union  $\mathbb A_1\cup  \cdots \cup \mathbb A_n$ with the following two properties for each $\mathbb A_i$:
\begin{enumerate}
    \item 
    for any two quasi-lines $\alpha\ne \beta\in \mathbb A_i$ we have $d(\alpha, \beta) \ge D$,
    \item
     the $(D+R)$--neighborhood of the union $\cup_{\gamma\in \mathbb A_i} \gamma$ contains $\mathcal X_1$.
\end{enumerate}
Indeed, by definition of $R$, any ball of radius $R$ intersects a quasi-line so for each  $\alpha\in \mathbb A$, there exists $\beta\in \mathbb A$ such that $D\le d(\alpha,\beta)\le D+2R$. Starting from a quasi-line $\gamma_1$, we inductively choose the quasi-lines which intersect the $(D+R)$--neighborhood of the already chosen ones, and by the axiom of choice,  a collection $\mathbb A_1$ of quasi-lines  containing $\gamma_1$ is obtained so that the properties  (1) and (2) are true.  The other collections $\mathbb A_i$ for $n\ge i\ge 1$ is obtained similarly from the other $n-1$ quasi-lines intersecting the $D$--neighborhood of $\gamma_1$. The property (2) guarantees    $\mathbb A\subseteq \cup_{1\le i\le n} \mathbb A_i$ from the choice of $R$. We do allow $\mathbb A_i\cap \mathbb A_j\ne \emptyset$, but any $\gamma\in \mathbb A$ would appear at most once in each $\mathbb A_i$.

By the defining property of $D$,   the collection $\mathbb{A}_{i}$ of   quasi-lines in the hyperbolic space $\mathcal{X}_1$ satisfies $diam(\pi_{\beta}(\alpha)) \le \theta$ for all $\alpha \neq \beta \in \mathbb{A}_i$. By Proposition~\ref{BBFDistanceThm}, $(\mathbb{A}_{i}, \pi_{\gamma})$ satisfies projection axioms with projection constant $\xi=\xi(\theta)$. For given $K>4\xi$, the quasi-tree of quasi-lines $(\mathcal C_K(\mathbb A_i), d_{\mathcal C_i})$ is a quasi-tree by \cite{BBF}. 
\end{proof}

\begin{proof}[Proof of Theorem~\ref{thm:main3}]
 Let $\mathcal{X}_i$ ($i = 1,2$) be the hyperbolic space constructed in Section~\ref{subsection:flipactions} with respect to $\mathcal{V}_i$. By Proposition~\ref{prop:qiehyperbolicspaces}, the admissible group $G$ admits a quasi-isometric embedding into $\mathcal{X}_1 \times \mathcal{X}_2$. Thus, to complete the proof of Theorem~\ref{thm:main3}, we only need to show that each hyperbolic space $\mathcal{X}_i$ is quasi-isometric embedded into a finite product of quasi-trees. We  give the proof for $\mathcal{X}_1$ and the proof for  $\mathcal{X}_2$ is symmetric.

Let $\rho \colon X \to T^0$ be the indexed map given by Remark~\ref{rem:indexfunction}. 
Let $\mathbb A_1, \dots, \mathbb A_n$ be the collection of quasi-lines given by Lemma~\ref{lem:partitionofA}.

Let $\hat{ \mathcal{X}_1} : = \cup_{v\in \mathcal V_1} \bar Y_v $. Since the $1$--neighborhood of $\hat{ \mathcal{X}_1}$ is $\mathcal{X}_1$, it suffices define a quasi-isometric embedding map from $\hat{ \mathcal{X}_1}$ to a finite product of quasi-trees.

We now define  a map  $$\Phi: \hat {\mathcal X_1} \to T \times \prod_{1\le i\le n} \mathcal C_K(\mathbb A_i),$$
where $T$ is the Bass-Serre tree of $G$.

Let $x\in \hat {\mathcal X_1} =\cup_{v\in \mathcal V_1} \bar Y_v$ and assume  $x\in \bar Y_v$. By the property~(2) of Lemma~\ref{lem:partitionofA}, we choose a point $\Phi_i(x) \in \cup_{\gamma\in \mathbb A_i} \gamma$ for each $1\le i\le n$ such that $d(x, \Phi_i(x))\le R+D$. Denote $\tilde R=R+D$.  Let $\Phi(x)=(\rho(x),\Phi_1(x),\cdots, \Phi_n(x))$.

We now verify that $\Phi$ is a quasi-isometric embedding. Since $d(x, (\Phi_i(x))\le \tilde R$ and $d(y, \Phi_i(y))\le \tilde R$, let     $K_0=K_0(\tilde R, \xi, \delta)$ be given by Corollary \ref{BBFDistanceCor} so that for $K>K_0$, the following distance formula   holds
$$
d_{\mathcal C_i}(\Phi_i(x), \Phi_i(y)) \sim_{K} \sum_{\gamma\in \mathbb A_i } 
[d_\gamma(x, y)]_{K}
$$
for each $1\le i\le n$. 
Therefore, 
$$
\begin{array}{rl}
d(\Phi(x), \Phi(y))&=d_T(\rho(x), \rho(y)) +\sum_{i=1}^n d_{\mathcal C_i}(\Phi_i(x), \Phi_i(y))\\
&\\
& \sim_K d_T(\rho(x), \rho(y)) +\sum_{i=1}^{n} \sum_{\gamma\in \mathbb A_i } 
[d_\gamma(x, y)]_{K}.
\end{array}
$$ 
Recall that $\mathbb A= \cup_{1\le i\le n} \mathbb A_i$ is a possibly non-disjoint union, but any quasi-line $\gamma$ with $[d_\gamma(x, y)]_{K}>0$ in the above sum appears at most once in each $\mathbb A_i$. We thus obtain
$$
\begin{array}{c}
d(\Phi(x), \Phi(y))  \sim_{K,n} d_T(\rho(x), \rho(y)) +  \sum_{\gamma\in \mathbb A } 
[d_\gamma(x, y)]_{K}
\end{array}
$$ 
which together with distance formula  (\ref{DistFormulaEQ}) for $\mathcal{X}_1$ concludes the proof of Theorem. 
\end{proof}

\section{Proper action on a finite product of quasi-trees}
\label{sec:properaction}
In this section, under a stronger assumption on  vertex groups as stated in Theorem \ref{thm:main4}, we shall promote the quasi-isometric embedding to be an orbital map of an action of  the admissible group on a (different) finite product of quasi-trees. 

By \cite[Induction 2.2]{BBF2}, if $H < G$ has finite index and acts on a finite product of quasi-trees, then so does $G$. We are thus free to pass to finite index subgroups in the proof.

Recall that $T^0=\mathcal V_1\cup\mathcal V_2$ where $\mathcal V_i$ consists of vertices in $T$ with pairwise even distances, and $ {\mathcal X}$ is the hyperbolic space constructed from $\mathcal V\in \{\mathcal V_1, \mathcal V_2\}$. By Lemma \ref{lem:index2subgroup}, let $\dot G < G$ be the subgroup of index at most $2$ preserving $\mathcal V_1$ and $\mathcal V_2$.


\subsection{Construct cone-off spaces: preparation}\label{SSConeoffSpace}
In this preparatory step, we first introduce   another hyperbolic space $\dot {\mathcal X}$ which is the cone-off of the previous hyperbolic space $\mathcal X$ over boundary lines of flat links.    We then embed $ {\mathcal X}$ into a   product of $\dot {\mathcal X}$ and a quasi-tree built from the set of binding lines from the flat links.
  
 
\begin{defn}[Hyperbolic cones] \cite[Part I, Ch. 5]{BH99}
For a line $\ell$ and a constant $r>0$, a \textit{hyperbolic $r$--cone}  denoted by $cone_r(\ell)$    is the quotient space of $\ell\times [0, r]$ by collapsing $\ell\times 0$ as a point called \textit{apex}. A metric   is endowed  on $cone_r(\ell)$  so that it  is isometric to the metric completion of the universal covering of   a closed   disk of radius $r$ in the real hyperbolic plane $\mathbb H^2$ punctured at the center.  

A \textit{hyperbolic multicones} of radius $r$ is the countable wedge of hyperbolic $r$--cones  with apex identified. 
\end{defn}

If $\xi$ is an isometry on $\ell$ then $\xi$ extends to a natural isometric  action on the hyperbolic cone $cone_r(\ell)$ which rotates around the apex and sends the radicals to radicals.  

Similar to the flat links, the link $Lk(w)$ for $w\in T^0$ determines a hyperbolic multicones of radius $r$ denoted by $Cones_r(w)$ so that the set of hyperbolic cones is bijective to the set of vertices adjacent to $w$. And $G_w=H_w\times Z(G_w)$ acts on $Cones_r(w)$ so that the center of $G_w$ rotates each hyperbolic cone around the apex and $G_w/Z(G_w)$ permutes the set of hyperbolic cones by the action of $G_w$ on $Lk(w)$.

{\bf Construction of the cone-off space $\dot{\mathcal  X}$.}
Let $\mathcal V\in \{\mathcal V_1, \mathcal V_2\}$ and $r>0$. Let $\dot{\mathcal  X}$ be the disjoint union of $\{\bar Y_v: v\in \mathcal{V}\}$ and hyperbolic multicones $\{Cones_r(w), w\in T^0 \setminus \mathcal{V}\}$ glued by isometry along  the boundary lines of $\cup_{v\in Lk(w)} \bar Y_v$ and those of   hyperbolic multicones $Cones(w)$. 

\begin{rem}
We note that $\dot {\mathcal X}$ is obtained from $\mathcal X$ by replacing each flat links by hyperbolic multicones.  However, the identification in $\dot{\mathcal  X}$ between boundary lines of $\bar Y_v$ and multicones $Cones(w)$ is only required to be isometric, while  the   $\mathbb R$--coordinates of the boundary lines  in constructing ${\mathcal  X}$ have to be matched up.
\end{rem}

We now give an alternative way to construct the cone-off space $\dot {\mathcal X}$, which shall be convenient in the sequel.

\textbf{Relatively hyperbolic structure of cone-off spaces.} Let  $\dot Y_v$ be  the disjoint union of $\bar Y_v$ and hyperbolic cones $cone_r(\ell)$ glued along boundary lines $\ell \in \mathbb L_v$.  If $E(\ell)$ denotes the stabilizer in $H_v$ of the boundary $\ell\in \mathbb L_v$, then $E(\ell)$ is virtually cyclic and almost malnormal. Since $\{E(\ell): \ell\in \mathbb L_v\}$ is $H_v$-finite by conjugacy, choose a complete set $\mathbb E_v$ of representatives.  By \cite{Bow12}, $H_v$ is hyperbolic relative to peripheral subgroups $\mathbb E_v$. 

Let    $\hat{\mathcal G}(H_v)$ be the  coned-off Cayley graph  (after choosing a finite generating set)  by adding a cone point for each peripheral coset of $E(\ell)\in \mathbb E_v$ and joining the cone point by half edges to each element in the  peripheral coset. The union of two half edges with two endpoints in a peripheral coset shall be referred to as an   \textit{peripheral edge}. See the relevant details in \cite{Farb}. 

By  \cite[Lemma 3.3]{Bow08}, \cite[Prop. 5.2]{Osin}, the action of $H_v$ on $\hat{\mathcal G}(H_v)$ is acylindrical. There exists a $H_v$-equivariant quasi-isometry between the   coned-off Cayley graph $\hat{\mathcal G}(H_v)$ and the coned-off space  $\dot Y_v$ which sends peripheral coset of $E(\ell)\in \mathbb E_v$ to $\ell$. Thus, the action of $H_v$ on $\dot Y_v$ is co-bounded and acylindrical as well.    
 
Alternatively, the cone-off space $\dot{\mathcal  X}$ could be obtained from the disjoint union of $\{\dot Y_v\}_{v\in \mathcal V}$ by identifying the apex of hyperbolic cones  from the same link $Lk(w)$ where ${w\in T^0\setminus \mathcal V}$. 

Since $\mathbb L_v$ has the bounded intersection property, by \cite[Corollary 5.39]{DGO}, for a sufficiently large constant $r$, the space {$\dot Y_v$} is a hyperbolic space with constant depending only on the original one. Thus, the space $\dot{\mathcal  X}$ is a hyperbolic space.

By Lemma \ref{lem:index2subgroup}, a subgroup $\dot G$ of index at most 2 in $G$ leaves invariant $\mathcal V_1$ and $\mathcal V_2$.  The following lemma is proved similarly as Lemma \ref{ActOnHypspaceLem}. 
\begin{lem}\label{ActOnConeoffHypspaceLem}
Fix a sufficiently large $r>0$. The space $\dot{\mathcal  X}$  is a $\delta$--hyperbolic space where $\delta>0$ only depends on the hyperbolicity constants of $\dot Y_v$  ($v\in \mathcal V$).

If $G_v=H_v\times Z(G_v)$ for every $v\in T^0$ and $G_e=Z(G_v)\times Z(G_w)$ for every edge $e=[v,w]\in T^1$, then  a subgroup $\dot G$ of index at most 2 in $G$   acts on  $\dot{\mathcal X}$    with   the following properties:
\begin{enumerate}
\item
for each $v\in \mathcal V$, the stabilizer of $\dot Y_v$ is isomorphic to $G_v$ and $H_v$ acts coboundedly on $\dot Y_v$, and
\item
    for each $w\in  T^0\setminus \mathcal V$, the stabilizer of the apex  of $Cones_r(w)$   is isomorphic to $G_w$ so that $H_w$ acts     on the set of hyperbolic cones by the action on the link $Lk(w)$ and  $Z(G_w)$ on acts by rotation on each hyperbolic cone. 
\end{enumerate}  
\end{lem}
 
From now on and until the end of the next subsection, we assume that   $G_v=H_v\times Z(G_v)$ for every $v\in T^0$ and $G_e=Z(G_v)\times Z(G_w)$ for every edge $e=[v,w]\in T^1$. After passage to a finite index subgroup of $G$, this assumption will be guaranteed by Corollary \ref{FiniteIndexAdmissibleCor} below, where  $H_v$ are assumed to omnipotent.

The remainder of this section is to relate the metric geometry of $(\dot{\mathcal  X},   d_{\dot{\mathcal  X}})$ and $(\mathcal X, d_{\mathcal X})$. This is achieved by a series of lemmas,  accumulating in Lemma \ref{QuasiGeodConeoffLem}.

Let $\pi_\ell$ denote the shortest projection to a quasi-line $\ell$ in $\mathcal X$ or in $\dot{\mathcal  X}$. To be clear, we use  $d_\ell(x, y)$ (resp. $\dot d_\ell(x, y)$) to denote  the $d_{\mathcal  X}$-diameter (resp. $d_{\dot{\mathcal  X}}$-diameter) of the projection of the points $x, y$ to $\ell$. 

Let $\mathbb L$ be the set of all binding lines in the flat links $Fl(w)$ over $w\in   T^0 \setminus \mathcal{V}$. Note that $\mathbb L$ is disjoint with the union $\cup_{v\in\mathcal V} \mathbb L_v$, although the binding lines in flat links $Fl(w)$ are   parallel  to the boundary lines of (the flat strips and) $\bar Y_v$ for $v\in Lk(w)$.   

We first give an analogue of Lemma \ref{YDistFormulaLem} for acylindrical action on coned-off spaces $\dot Y_v$.

Up the equivariant quasi-isometry mentioned above, we shall identify $(\bar Y_v,d)$ with  the Cayley graph of $H_v$, and   $(\dot Y_v,\dot d)$ with the  coned-off Cayley graph     $\hat{\mathcal G}(H_v)$, and    $\mathbb L_v$ with the collection of left peripheral cosets of  $\mathbb E_v$.

A geodesic edge path $\beta$  in the coned-off Cayley graph $\dot Y_v$ is \textit{$K$-bounded} for $K>0$ if every peripheral edge  has two endpoints within   $\bar Y_v$-distance   at most $K$.

By definition,  a geodesic     $\beta=[x,y]$ can be subdivided into maximal  $K$-bounded non-trivial segments   $\alpha_i$ ($0\le i\le n$) separated by peripheral edges $e_j$ ($0\le j\le m$) where $d_{\bar Y_v}((e_j)_-, (e_j)_+)>K$. It is possible that $n=0$:   $\beta$ consists of only peripheral edges.   

Define $$|\beta|_K:=\sum_{0\le i\le n} [Len(\alpha_i)]_K,$$  which   sums up the lengths of $K$-bounded subpaths of length at least $K$. It is possible that $n=0$, so $|\beta|_K=0$.  Define $d_{\dot Y_v}^{K}(x, y)$ to be the maximum of $|\beta|_K$ over all relative geodesics $\beta$ between $x,y$. Thus, $d_{\dot Y_v}^{K}(x, y)$ is $H_v$-invariant. We remark that the ``thick" distance $d_{\dot Y_v}^{K}(x, y)$ is inspired by the corresponding \cite[Def. 4.8]{BBF2} in mapping class groups.

\begin{lem}
For any sufficiently large $K>0$,  
\begin{equation}\label{ConeoffYvDistFormulaEQ2}
d_{\bar Y_v}(x, y) \sim_K d_{\dot Y_v}^{K}(x, y) + \sum_{\gamma\in \mathbb L_v } 
[ d_\gamma(x, y)]_K  
\end{equation}
\end{lem}
\begin{proof}
Let $\beta$ be a geodesic between $x,y$ in $\dot Y_v$ so that $d_{\dot Y_v}(x, y)=Len(\beta)$.  We obtain the \textit{lifted path} $\hat \beta$ by replacing each peripheral edge $e$ with endpoints in a peripheral coset $\ell_e$ by a geodesic in $\bar Y_v$ with same endpoints. The well-known fact (\cite{DS05}, \cite{GP16}) that $\beta$ is a uniform quasi-geodesic gives the so-called distance formula \cite[Theorem 0.1]{S13}: for any $K\gg 0$,     $$d_{\bar Y_v}(x, y) \sim_K d_{\dot Y_v}(x, y) + \sum_{\gamma\in \mathbb L_v } 
[ d_\gamma(x, y)]_K.$$
Indeed, the additional ingredient (to quasi-geodesicity of $\hat\beta$) is the following fact: 
  $d_\gamma(x, y)\ge K$ if and only if $\beta$ contains a peripheral edge $e$ with two endpoints in $\gamma$ with $d_{\bar Y_v}(e_-, e_+)>K'$, where $K$ depends on $K'$ and vice versa.  

By definition,  $Len(\beta)$ differs from $|\beta|_{K}=d_{\dot Y_v}^{K}(x, y)$ in the sum of lengths of at most $(m+2)$ segments  $\alpha_i$    and $(m+1)$ edges $e_j$, where $Len(\alpha_i)<K$ and $d_{\bar Y_v}((e_j)_-,(e_j)_+)\ge K$.  Therefore, we can replace $d_{\dot Y_v}(x, y)$ with $d_{\dot Y_v}^{K}(x, y)$ by worsening  the multiplicative constant in the above distance formula. This shows (\ref{ConeoffYvDistFormulaEQ2}).
\end{proof}


\begin{lem} \label{ConeofYDistFormulaLem}
For any sufficiently large $K>\max\{4\xi, \theta\}$,  there exist  a  $H_v$--finite collection $\mathbb A_v$  of quasi-lines    in $\dot Y_v$  and a constant $N=N(K, \delta, \mathbb A_v)>0$, such that  for any two vertices $x, y\in  \dot Y_v$, the following holds

\begin{equation}\label{ConeoffYvDistFormulaEQ}
d_{\dot Y_v}^{K}(x, y) \sim_N \sum_{\gamma\in \mathbb A_v } 
[\dot d_\gamma(x, y)]_K  
\end{equation}
\end{lem}
\begin{proof}
For simplicity, we suppress indices $v$ from $H_v, Y_v, \dot Y_v, \mathbb L_v$ in the proof.

Let $S$ be the   set of elements $h\in H$  so that $\dot d(1,h)= K$ and some geodesic $[1,h]$ in $\dot Y$ is $K$-bounded.  The definition of $K$-boundedness implies   that  $S$ is a finite set, since $d(1, h)\le K^2$ and the word metric $d$ is proper.

Let $\beta$ be a geodesic path between $x, y$ so that $d_{\dot Y_v}^{K}(x, y)=|\beta|_K$. Let $\{\alpha_i: 0\le i\le n\}$ be the set of maximal $K$-bounded segments of   $\beta$ in $\dot Y$. Then 
\begin{equation}\label{dotdxyDistEQ}
\begin{array}{rl}
|\beta|_K  = \; \sum_{i=0}^n [Len_{\dot Y}(\alpha_i)]_K.
\end{array}
\end{equation}


We now follow the argument in the proof of Lemma \ref{YDistFormulaLem} to produce the collection of quasi-lines $\mathbb A$ to approximate $Len_{\dot Y}(\alpha_i)$.  

The argument below considers every $\alpha_i$ in the above sum of (\ref{dotdxyDistEQ}), but for simplicity, denote $\alpha=\alpha_i$  with endpoints $\alpha_-,\alpha_+$. 

Since the action of $H$ on $\dot Y$ is acylindrical, consider the set $\tilde S$ of loxodromic elements $hf$ on $\dot Y$ where $h\in S$ and $f\in F$ is provided by Lemma \ref{ExtensionLem}. Since $S$ is finite, $\tilde S$ is finite as well. Let $\mathbb A$ be the set of all $H$--translated axis of $hf\in \tilde S$.  

If $Len_{\dot Y}(\alpha)\ge K$, we subdivide further $\alpha=\gamma_0\cdot \gamma_1\cdots\gamma_n$, so that subsegments $\gamma_i$ $(0\le i< n)$ have $\dot Y$-length $K$ and the last one $\gamma_n$  have $\dot Y$-length less than $K$. Each $\gamma_i$ is $K$-bounded since the $K$-boundedness is preserved by taking subpaths. Similarly as in Lemma \ref{YDistFormulaLem}, a set of axes of loxodromic elements $\gamma_i f$ is then produced to give the upper bound of $Len_{\dot Y}(\alpha)$.  The lower bound uses a certain local finiteness of $\mathbb A$ due to the proper action in Lemma \ref{YDistFormulaLem},  which  is now guaranteed  by Corollary \ref{LocSegmFiniteCor} in an acylindrical action.  Thus, 
\begin{equation}\label{alphaLenEQ}
\frac{1}{N} \sum_{\gamma\in \mathbb A} 
[d_\gamma(\alpha_-, \alpha_+)]_K\le Len_{\dot Y}(\alpha)\le   2\sum_{\gamma\in \mathbb A } 
[d_\gamma(\alpha_-, \alpha_+)]_K + 2K
\end{equation}
where the constant $N$ depends on $K,\delta, \mathbb A$.

Recall that each $\gamma\in \mathbb A$ is a uniform quasi-line, i.e.: a $(\lambda,c)$-quasi-geodesic in $\dot Y$.  Since $\alpha=\alpha_i$ is a subsegment of the geodesic $\beta=[x,y]$,  the hyperbolicity of $\dot Y$ implies that $\pi_\gamma(\alpha_-)$ and $\pi_\gamma(\alpha_+)$ are contained in a $C$-neighborhood of $\pi_\gamma(x)$ and  $\pi_\gamma(y)$ respectively, where $C$ depends on $\lambda, c, \delta$. Thus, we have 
\begin{equation}\label{MoToEndptsEQ}
[\dot d_\gamma(\alpha_-, \alpha_+)]_K\asymp_{K,C} [\dot d_\gamma(x,y)]_K    
\end{equation} for every $\gamma\in \mathbb A$. So for every $\alpha=\alpha_i$, we obtained (\ref{alphaLenEQ}) and (\ref{MoToEndptsEQ}),  and the conclusion thus follows from (\ref{dotdxyDistEQ}).
\end{proof}


Recall that a geodesic $\gamma=[x,y]$ in $\dot {\mathcal X}$ is the union of a sequence of maximal geodesics $\beta_v$ in $\dot Y_v$'s. As    $d_{\dot Y_v}^{K}(x, y)$ is related with $d_{\dot Y_v}(x, y)$, we define 
\begin{equation}\label{ThickDistEQ}
d_{\dot {\mathcal X}}^{K}(x, y):=\sum_{v} d_{\dot Y_v}^{K}((\beta_v)_-, (\beta_v)+).    
\end{equation}
Since $d_{\dot Y_v}^{K}(\cdot, \cdot)$ is $H_v$-invariant, we see that $d_{\dot {\mathcal X}}^{K}$ is $\dot G$-invariant (even though it is not a metric anymore).

\begin{lem}\label{QuasiGeodConeoffLem}
There exists $K_0>0$ such that for any two points $x, y\in \mathcal X$ and $K>K_0$, we have
\begin{equation}\label{QuasiGeodFormulaEQ}
d_{\mathcal X}(x,y)\sim_K  d_{\dot{\mathcal X}}^K(x,y)+\sum_{\ell\in \mathbb L} [d_\gamma(x, y)]_K
\end{equation}

\end{lem}
Recall that the notation  $A \sim_K B$ means that $A$   equals $B$  up to multiplicative and additive constants depending on $K$. 

\begin{proof}
Let $\gamma=[x,y]$ be a $\dot {\mathcal X}$-geodesic with endpoints $x, y\in \mathcal X$. Let us consider the generic case that $\rho(x)\ne \rho(y)$; the case $\rho(x)\ne \rho(y)$ is much easier and left to the reader.

Assume that $x, y\in \mathcal X$ are not contained in any hyperbolic $r$-cones, up to moving $x, y$ with a distance at most $r$.  We can then write the geodesic $\gamma$ as the following union 
\begin{equation}\label{PathDecompEQ}
\gamma=\left(\cup_{v\in \mathcal V\cap [\rho(x), \rho(y)]} \beta_v\right) \bigcup \left(\cup_{w\in (T^0\setminus \mathcal V)\cap [\rho(x), \rho(y)]} c_w\right)    
\end{equation}
where  $\beta_v$ is a geodesic in the cone-off space $\dot Y_v$ with endpoints   on boundary lines of $Y_v\subset \dot Y_v$,  and   $c_w$ is contained in  the hyperbolic multicones $Cones_r(w)$ passing the apex. Thus, $Len(c_w)$ has  length $2r$.

We replace $c:=c_w$ by an $\mathcal X$-geodesic $c'$ with the same endpoints $c_-, c_+$: $c'$ is contained  in the corresponding flat links, whose binding line is denoted by $\ell_c$. Thus, $c'\subset N_1(\ell_c)$.  This replacement becomes non-unique when different $c$'s have overlap (in the subspace $\bar Y_v\subset \dot{ Y_v}$). However, the bounded  intersection of $\mathbb L$ gives a uniform upper bound on the overlap. Let $K$ be any constant sufficiently bigger than this   bound.  We then number those subpaths $c$ of $\gamma$ with $d_{\mathcal X}(c_-,c_+)>K$ in a fixed order (eg. from left to right): $c_1, \dots, c_n$. Up to bounded modifications, we obtain a well-defined  notion of \textit{lifted path} $\hat \gamma$ with same endpoints of $\gamma$. 

Observe that $\hat \gamma$ is a uniform quasi-geodesic. This is well-known and one proof  proceeds as follows: since the set of binding lines $\mathbb L$ has bounded intersection in $\mathcal X$, the above construction implies that $\hat \gamma$  is an efficient semi-polygonal path in the sense of   \cite[6, Section 7]{Bow12}. The observation then follows as a consequence of  \cite[Lemma 7.3]{Bow12}.  

If $d_{\mathcal X}(c_-,c_+)> K$ so $c'\subset N_1(\ell_c)\cap \hat \gamma$, then $d_{\mathcal X}(c_-,c_+)\asymp [d_{\ell_{c}}(x, y)]_K$   follows from the quasi-geodesicity of $\hat \gamma$ and the  hyperbolicity of $\mathcal X$. We incorporate the $c_w$'s in Eq. (\ref{PathDecompEQ}) with $d_{\mathcal X}((c_w)_-,(c_w)_+)\le K$    into the multiplicative constant, and thus the following holds
$$Len_{\mathcal X}(\hat \gamma) \sim_K \sum_{v\in \mathcal V\cap [\rho(x), \rho(y)]} Len_{\mathcal X}(\beta_v)+\sum_{i=1}^n [d_{\ell_{c_i}}(x, y)]_K.$$
Combining the formula  (\ref{ConeoffYvDistFormulaEQ2}) and   (\ref{ThickDistEQ}), we get
$$Len_{\mathcal X}(\beta_v)\sim |\beta_v|_K+\sum_{\ell\in \mathbb L_v} [d_{\gamma}((\beta_v)_-, (\beta_v)_+)]_K.$$

Using the local finiteness and bounded intersection $\mathbb L$, for each $\ell_{c_i}$ with $[d_{\gamma_{c_i}}(x, y)]_K>0$, there are only finitely many $\ell\in \mathbb L$ such that $[d_{\gamma_{\ell}}(x, y)]_K>0$. Hence, by worsening the multiplicative constant, we have 
$$Len_{\mathcal X}(\hat \gamma) \asymp_K \sum_{v\in \mathcal V\cap [\rho(x), \rho(y)]} |\beta_v|_K+\sum_{\ell\in \mathbb L} [d_{\ell}(x, y)]_K.$$

Recall that $Len_{\mathcal X}(\hat \gamma)\sim d_{\mathcal X}(x,y)$ from the quasi-geodesic $\hat \gamma$. Then the desired (\ref{QuasiGeodFormulaEQ}) follows.
\end{proof}

Recall that $\mathbb L$ is the set of all binding lines in   $Fl(w)$ over $w\in   T^0 \setminus \mathcal{V}$. Since $\mathbb L$ has bounded projection property in $\mathcal X$,   it satisfies the projection axioms  and we can construct the quasi-tree $\mathcal C_K(\mathbb L)$ for $K\gg 0$, equipped with the length metric $d_{\mathcal C}$.  

By the following result, we shall be reduced to embed the action of $\dot G$ on $(\dot{\mathcal  X}, d_{\dot{\mathcal X}}^K)$ into finite product of quasi-trees. We warn the reader that $d_{\dot{\mathcal X}}^K$ is not a metric on $\dot{\mathcal  X}$. However, we can still talk about the product space $\dot{\mathcal  X} \times \mathcal C_K(\mathbb L)$ equipped with the symmetric non-negative function $d_{\dot{\mathcal X}}^K\times d_{\mathcal C}$.  The quasi-isometric embedding of $\mathcal X$ into the product only means the coarse bilipschitz inequality.

\begin{prop}\label{QIembedProp}
For any $K\gg K_0$, there exists a  $\dot G$--equivariant quasi-isometric embedding from ${\mathcal  X}$ to the product $(\dot{\mathcal  X}  \times \mathcal C_K(\mathbb L), d_{\dot{\mathcal X}}^K\times d_{\mathcal C})$. 
\end{prop}
\begin{proof}
We first define the map $\Phi=\Phi_1\times \Phi_2:\mathcal  X\to \dot{\mathcal  X} \times \mathcal C_K(\mathbb L) $.   

If $x\in \bar Y_v$  define $\Phi_1(x)=x$ in $\dot{\mathcal  X}$. To define $\Phi_2(x)$, choose a (non-unique) closest quasi-line $\ell$ in $\mathbb L_v$ to $x$  and   define $\Phi_2(x)=\pi_\ell(x)$. The choice of $\ell$ is not important and the important thing is the distance from $o$ to any such chosen $\ell$ is uniformly bounded, thanks to co-compact action of $H_v$ on $\bar Y_v$.   Extend by $G$--equivariance    $\Phi_2(gx)=g\pi_\ell(x)$ for all $g\in G$. 

If $x$ lies in the flat links $Fl(w)$   define  $\Phi_1(x)$ to be the apex of $Cones_r(w)$ and $\Phi_2(x)=\pi_\ell(x)$ where $\ell$ is the binding line of $Fl(w)$.  

By Corollary \ref{BBFDistanceCor}, we have that 
$$\sum_{\ell\in \mathbb L} [d_\gamma(x, y)]_K \sim_K d_{\mathcal C}(\Phi_2(x),\Phi_2(y)).
$$ 
The quasi-isometric embedding then follows from the distance formula (\ref{QuasiGeodFormulaEQ}) in Lemma \ref{QuasiGeodConeoffLem}.
\end{proof}

\subsection{Construct the collection of quasi-lines in $\dot {\mathcal X}$}
\label{SSQuasilines} 

The goal of this subsection is to introduce a  collection $\mathbb A$ of quasi-lines so that a    distance formula holds for $(\dot{\mathcal X}, d_{\dot{\mathcal X}}^K)$.

Recall that $\dot {\mathcal X}$ is the hyperbolic cone-off space constructed from $\mathcal V\in \{\mathcal V_1, \mathcal V_2\}$. According to Proposition \ref{QIembedProp}, we are working in the cone-off space $\dot {\mathcal X}$  endowed with length   metric $\dot d$. In particular, all quasi-lines are understood with this metric, and boundary lines $\mathbb L_v$ of $\bar Y_v$ are of bounded diameter so  are not quasi-lines anymore in $\dot {\mathcal X}$.
 
First of all, let us fix the constants used in the sequel.

Recall that for every $v\in\mathcal V$, the cone-off space $\dot Y_v$ admits a co-bounded and acylindrical action of $H_v$. Thus, {\bf when talking about quasi-lines, we follow the Convention \ref{ConvQuasiLine}}: quasi-lines are $(\lambda, c)$--quasi-geodesics in $\dot {\mathcal X_i}$ and  $\dot Y_v$'s (isometrically embedded into the former), where $\lambda, c>0$ are given by Lemma \ref{ExtensionLem} applied to those actions of $H_v$ on $\dot Y_v$.   
 
If $\gamma$ is a quasi-line in $\dot {\mathcal X}$, denote by $\dot d_\gamma(x,y)$ the $\dot d$--diameter of the shortest projection of $x, y\in \dot {\mathcal X}$ to $\gamma$ in $\dot {\mathcal X}$. 
By Lemma \ref{ActOnConeoffHypspaceLem}, $\dot {\mathcal X}$ is $\delta$--hyperbolic for a constant $\delta>0$. The coning-off construction is crucial to obtain the uniform constant $\theta$ in the next lemma. 
\begin{lem}\label{CstThetaLem}
There exists a constant $\theta>0$ depending on $\delta$   (and also $\lambda, c$) with the following property: for any ($(\lambda,c)$--)quasi-lines  $\alpha$ in $\dot Y_v$ and $\beta$ in $\dot Y_{v'}$ with $v\ne v'\in \mathcal V$  we have  $diam_{\dot {\mathcal X_1}} (\pi_\beta (\alpha)) \le \theta$.
\end{lem}
\begin{proof}
By the construction of $\dot {\mathcal X}$, any geodesic from $\alpha$ to $\beta$   has to pass through the apex between $\dot Y_v$ and $\dot Y_{v'}$, and thus the shortest projection $\pi_\beta (\alpha)$ is contained in the projection of the apex to $\beta$. By hyperbolicity, there exists a constant $\theta$ depending only on $\lambda, c, \delta$ such that the diameter of the projection of any point to every quasi-line is bounded above by $\theta$. The conclusion then follows.  
\end{proof}

Fix $K>\max\{4\xi, \theta\}$. For each $v\in \mathcal V$,  there exist  a {$H_v$-finite collection of quasi-lines} $ \mathbb A_v$    in $\dot Y_v$ and a constant $N=N(\mathbb A_v, K, \delta)$ such that (\ref{ConeoffYvDistFormulaEQ}) holds.

Since $\dot G$ preserves $\mathcal V_1$ and $\mathcal V_2$,   by Lemma \ref{ActOnConeoffHypspaceLem}, there are only finitely many $\dot G$--orbits in $\{\dot Y_v: v\in \mathcal V\}$, so  we can assume furthermore $\mathbb A_w=g\mathbb A_v$ if $w=gv$ for $g\in \dot G$.  Then the collection $\mathbb A= \cup_{v\in\mathcal  V} \mathbb A_v$ is a  $\dot G$--invariant collection of quasi-lines in $\dot{\mathcal X}$.

Recall that $r$ is the radius of the multicones in constructing $\dot {\mathcal X_i}$ for $i=1, 2$, and $T$ is the Bass-Serre tree for admissible group $G$.  
\begin{prop}\label{ConeoffDistFProp}
For any $x, y\in \dot{\mathcal X}$, the following holds   
\begin{equation}\label{ConeoffDistFormulaEQ}
\begin{array}{cc}
  d_{\dot{\mathcal X}}^K (x, y) \;\sim_{N,r} \; \sum_{\gamma\in \mathbb A } [\dot d_\gamma(x, y)]_{K} +   d_T(\rho(x), \rho(y)).
\end{array}
\end{equation} 
\end{prop}
\begin{proof}
The proof proceeds similarly as that of Proposition \ref{prop:distanceformulaX_1}, so  only the differences are spelled out. Assume that $x, y\in \dot{\mathcal X}$ are not in any hyperbolic multicones.  We can then write the geodesic $[x,y]$ as in (\ref{PathDecompEQ}) and keep notations there. 

By the formula (\ref{ConeoffYvDistFormulaEQ}),   summing up the lengths of geodesics $\gamma_v$ in $\dot Y_v$   yields  the upper bound in (\ref{ConeoffDistFormulaEQ}). The term $d_T(\rho(x), \rho(y))$ appears since the notation $\sim$ involves additive errors. 

By the choice of $K>\theta$ and Lemma \ref{CstThetaLem},  we have $[\dot d_\gamma (x, y)]_K=0$ for any quasi-line $\gamma$ in $\mathcal V\setminus [\rho(x), \rho(y)]^0$. The lower bound of  (\ref{ConeoffDistFormulaEQ}) is obtained as well by summing up distances in the formula (\ref{ConeoffYvDistFormulaEQ}). A term $r \cdot d_T(\rho(x), \rho(y))$ is added, since $[x,y]$ goes through $d_T(\rho(x), \rho(y))/2$ hyperbolic cones with radius $r$ and each $c_w$ is of length $2r$.
\end{proof}
\subsection{Reassembling finite index vertex groups}
By Bass-Serre theory, the finite index subgroup $\dot G<G$ from Lemma \ref{lem:index2subgroup} acts on the Bass-Serre tree of $G$ and can be represented as    a finite graph $\mathcal G=T/\dot G$ of groups where  the vertex subgroups are isomorphic to those of $G$.   

Let $e$ be an oriented  edge in $\mathcal{G}$ from $e_-$ to $e_+$ (it is possible that $e_- = e_+$ because $e$ could be a loop) and $\bar e$ be the oriented edge with reversed orientation. A collection of finite index subgroups $\{G_e'<G_e, G_v'<G_v: v\in \mathcal G^0, e\in \mathcal G^1\}$ is called \textit{compatible} if whenever $v = e_-$, we have
$$G_v \cap G_e' = G_v' \cap G_e.$$ 
We shall make use of \cite[Theorem 7.51]{DK18} to obtain a finite index subgroup in $\dot G$ from a compatible collection of finite index subgroups. For this purpose, we assume the quotient $H_v$ of each vertex group $G_v$ for $v\in \mathcal G^0$ is omnipotent in the sense of Wise. 

In a group two elements  are \textit{independent} if they do not have conjugate powers.   (see Definition~3.2 in \cite{Wise00}).

\begin{defn}
\label{defn:omnipotent}
A group $H$ is \textit{omnipotent} if for any set of pairwise independent   elements $\{h_1,\cdots, h_r\}$ ($r\ge 1$) there is a number $p\ge 1$ such that for every choice of  positive natural numbers $\{n_1,\cdots, n_r\}$ there is a finite quotient $H\to \hat H$  such that $\hat h_i$ has order $n_ip$ for each $i$.
\end{defn}
\begin{rem}
If $H$ is hyperbolic, two  loxodromic elements $h, h'$   are usually called  \textit{independent}  if the collection of $H$--translated quasi-axis of $h, h'$ has the bounded projection property. When $H$ is torsion-free, it is equivalent to the notion of independence in the above sense.
\end{rem}


Let $g$ be a loxodromic element in a hyperbolic group and  $E(g)$ be  the maximal elementary group containing $\langle g\rangle$.  
By  \cite[Ch. II  Theorem 6.12]{BH99}, $G_v$ contains a subgroup $K_v$ intersecting trivially with $Z(G_v)$ so that the direct product $K_v\times Z(G_v)$ is a finite index subgroup. Thus, the image of $K_v$ in $G_v/Z(G_v)$ is of finite index in $H_v$ and $K_v$ acts geometrically on  hyperbolic spaces $\bar Y_v$.

The following result will be used in the next subsection  to obtain desired finite index subgroups.

\begin{lem}\label{FindexSubgrpsLem}
Let $\{\dot K_v<K_v: v\in \mathcal G^0 \}$ be a   collection of finite index subgroups. Then there exists a compatible collection of finite index subgroups $\{G_e'<G_e, G_v'<G_v: v\in \mathcal G^0, e\in \mathcal G^1\}$  such that $G_v'=\ddot K_v \times \mathbb Z$ is of finite index in $\ddot K_v\times Z(G_v)$ for each $v\in \mathcal G^0$, where $\ddot K_v$ is of finite index in $\dot K_v$. 
\end{lem}
\begin{proof}
Let $e$ be an oriented  edge in $\mathcal{G}$ from $e_-$ to $e_+$ (it is possible that $e_- = e_+$) and $\bar e$ be the oriented edge with reversed orientation.
 
If $v=e_-$, then the abelian group $K_v\cap G_e$ is  a nontrivial cyclic group   contained in a maximal elementary $E(b_e)$ in $K_v$ where $b_e$ is a primitive loxodromic element.  Similarly, for $w=e_+$, let $b_{\bar e}\in K_w$ be a primitive loxodromic element in $E(b_{\bar e})$ containing  $K_{w}\cap G_{e}$. Then $b_{ e}$ and $b_{\bar e}$ preserve  two lines respectively which are orthogonal in the Euclidean plane $F_e$ and thus generate  an abelian  group $\hat G_e:=\langle b_{e}, b_{\bar e}\rangle$ of rank 2 so that  $G_e\subset \hat G_e$ is of finite index.

 
Consider the collection of all oriented edges  $e_1, \dots, e_r$ in $\mathcal{G}^1$ such that  $(e_i)_-=v$. Let  $\{b_{e_1},\dots,b_{e_r}\}$ be the set of primitive loxodromic elements in $K_v$ obtained as above in correspondence  with  $\{e_1, \dots, e_r\}$. Note that $\{b_{e_1},\dots,b_{e_r}\}$ are pairwise independent in $H_v$. This follows from the item (3) in Definition \ref{defn:admissible} of admissible group, otherwise $G_{e_i}$ and $G_{e_j}$ would be commensurable for $e_i\ne e_j$.

By the finite index of $G_e$ in $\hat G_e$, there exists a set of powers of $b_{e_i}$'s  in $\dot K_v\cap G_{e_i}$ denoted by $\{h_{e_1}, \cdots, h_{e_r}\}$. Since $\{h_{e_1}, \cdots, h_{e_r}\}$ are still pairwise independent in $H_v$,  the omnipotence of $H_v$ gives       the constant $p_v$ by Definition~\ref{defn:omnipotent}. Let $$s =  \prod_{v\in\mathcal G^0} p_v$$

Define $n_i = \frac{s}{p_v}$ with $i \in \{1, \dots r\}$. By the omnipotence of $H_v$ and restricting to $\dot K_v\subset H_v$,  there is  a finite index subgroup $\ddot K_v$ of $\dot K_v$ such that $h_{e_i}^{p_vn_i}=h_{e_i}^s \in \ddot{K}_v$.

For each vertex $v$ in $\mathcal{G}$ and for each edge $e$ in  $\mathcal{G}$, we define 
\[
G'_v := \ddot{K}_{v} \times \langle h_{\bar e}^{s} \rangle
\]
and 
\[
G'_e := \langle h_{e}^{s} \rangle \times \langle h_{\bar e}^{s} \rangle =s\Z \times s\Z
\]

To conclude the proof, it remains to note the collection $\{G'_v, G'_e \,\,| \, v \in \mathcal{G}^{0}, e \in \mathcal{G}^{1} \}$ is   compatible.
  It is obvious that $G'_{e_i} \subset G'_v$, so $G'_{e_i} \le G'_{v} \cap G_{e_i}$. Conversely, $G'_{v} \cap G_{e_i} \subset G'_{v} \cap\hat  G_{e_i}\subset  (\ddot{K}_{v}\cap \langle h_{e}^{s} \rangle) \times \langle h_{\bar e}^{s} \rangle \subset  G'_{e_i}$.   
\end{proof}

\begin{cor}\label{FiniteIndexAdmissibleCor}
There is a finite index admissible group $G_0<\dot G<G$ in the sense of  Definition \ref{defn:admissible} where every vertex group are direct products of a hyperbolic group and $\mathbb Z$.
\end{cor}
\begin{proof}
By \cite[Theorem 7.50]{DK18}, the compatible collection of finite index subgroups from Lemma \ref{FindexSubgrpsLem} determines a finite index group $G_0<\dot G<G$. Indeed,   $G_0$ is the fundamental group of a finite covering space which are obtained from finite many copies of finite coverings in correspondence to $G_v'<G_v$ glueeing along edge spaces in correspondence to $G_e'$. Thus, $G_0$ splits over edge groups $\mathbb Z^2$ as a finite graph of groups where the vertex groups are conjugates of $G_v'=\ddot K_v\times \mathbb Z$. In view of Definition of  \ref{defn:admissible}, it suffices to certify the non-commensurable edge groups adjacent to the same vertex group. This follows from   the above proof of Lemma \ref{FindexSubgrpsLem}, where the edge groups are direct products of $\mathbb Z$ with pairwise independent loxodromic elements. Thus, different edge groups are  not commensurable.      
\end{proof}

\subsection{Partition $\mathbb A$ into sub-collections with good projection constants: completion of the proof}\label{SSPartitonProperly}
With purpose to prove Theorem \ref{thm:main4}, it suffices to prove property (QT) for a finite index subgroup of $G$. By Corollary \ref{FiniteIndexAdmissibleCor}, we can assume that the CKA flip action of $G$ on $X$ satisfies that every vertex group are direct products $K_v\times \mathbb Z$ where $K_v$ is of finite index in $H_v$.  Since $H_v$ is omnipotent and then residually finite,  without loss of generality we can assume that $K_v$ is torsion-free. 

Since the assumption of Lemma \ref{ActOnConeoffHypspaceLem} is fulfilled, the results in Sections \ref{SSConeoffSpace} and \ref{SSQuasilines}  hold: a finite index at most 2  subgroup $\dot G<G$ acts on the cone-off spaces $\dot{\mathcal X_i}$ for $i=1,2$ with distance formula. 

Let $\mathcal V\in \{\mathcal V_1, \mathcal V_2\}$.
Let us recall the data we have now:
\begin{enumerate}
    \item 
    For every $v\in\mathcal V$, $\mathbb A_v$  is a  $K_v$-finite   collection of quasi-lines   in $\dot Y_v$
so that the distance formula (\ref{ConeoffYvDistFormulaEQ}) holds for $\dot Y_v$. (Lemma \ref{ConeofYDistFormulaLem})
\item
Let $\mathbb A=\cup_{v\in\mathcal V}\mathbb A_v$ be the $\dot G$-invariant collection of quasi-lines  so that the formula (\ref{ConeoffDistFormulaEQ}) holds. (Proposition \ref{ConeoffDistFProp})
\end{enumerate} 

The first step is passing to a further finite index subgroup $\dot K_v$ of $K_v$ so that $\mathbb A_v$ is  partitioned into $\dot {K_v}$-invariant sub-collections with    projection constants $\xi$. It follows closely the argument in  \cite{BBF2} which is presented below for completeness.

{\bf The constants $\theta$ and $\xi$:}  The constant $\theta>0$ is chosen so that it satisfies Proposition \ref{FiniteDblCosetsProp} and Lemma \ref{CstThetaLem} simultaneously.  Then $\xi= \xi(\theta)$ is given by Proposition \ref{BBFDistanceThm}.

\begin{lem}\label{FirstPartitionLem}
Let $\mathbb A_v$ be  a $K_v$-finite collection of quasi-lines obtained as above by Lemma  \ref{YDistFormulaLem}. Then there exists a finite index subgroup $\dot{K_v} <K_v$ such that       any two distinct quasi-lines in the same $\dot K_v$-orbit  have $\theta$-bounded projection.   
\end{lem}
\begin{proof}
By construction, the quasi-lines in   $\mathbb A_v$ are quasi-axis of loxodromic elements whose maximal elementary groups are virtually cyclic. Recalling that $K_v$ is torsion-free,   the maximal elementary group is cyclic and thus $E(g)$ is the centralizer  $C(g):=\{h\in K_v: hg=gh\}$ of $g$.  By \cite[Lemma 2.1]{BBF2},  since $K_v$ is   residually finite, then  the centralizer   of any element $g\in K_v$ is   separable, i.e.  the intersection of all finite index subgroups containing $C(g)$.

Proposition \ref{FiniteDblCosetsProp} implies that  $E:=\{h\in K_v: diam(\pi_\gamma(h\gamma))\ge  \theta\}$ consists of finite double $C(g)$-cosets. Since $C(g)$ is separable, we use the remark after Lemma~2.1 in \cite{BBF2} to get a finite index $\dot K_v<K_v$ such that $E\cap \dot K_v=\emptyset$. The proof is complete.
\end{proof}

The next step is re-grouping  appropriately  the    collections of quasi-lines  $\cup_{v\in \mathcal V} \mathbb A_v$ in Lemma \ref{FirstPartitionLem}. 

By \cite[Theorem 7.51]{DK18}, the compatible collection of finite index subgroups $\ddot K_v < \dot K_v$ from Lemma \ref{FindexSubgrpsLem} determines a finite index group $G_0<\dot G<G$  
such that
$G_0 \cap G_v = G_v', G_0 \cap G_e = G_e'$
 and $G_v'\subset \ddot K_v\times Z(G_v)\subset \dot K_v\times Z(G_v)$ for every vertex $v$ and edge $e$.

By Bass-Serre theory, $G_0$ acts on the Bass-Serre tree $T$ of $G$ with finitely many vertex orbits. To be precise, let $\{v_0, \cdots, v_m\}$ be the full set of vertex representatives. 

Since for each $1\le i\le m$, $\ddot K_{v_i}\subset \dot K_{v_i}$ is of finite index,   Lemma \ref{FirstPartitionLem} implies that  $\mathbb A_{v_i}$ consists of finitely many $\ddot K_v$-orbits, say $\check{\mathbb A}_{i}^j (1\le j\le l_i)$, 
$$
\mathbb A_{v_i}=\cup_{j=1}^{l_i} \check{\mathbb A}_{i}^j,
$$
each of which satisfies projection axioms with projection constant $\xi$. 

Recall that $G_0\subset \dot G$ acts on $\mathcal X_1$ and $\mathcal X_2$. We now set $\mathbb A_{ij}: =\cup_{g\in G_0} g\check{\mathbb A}_{i}^j$ so we have
$$\mathbb A=\cup_{i=1}^m\cup_{j=1}^{l_i} \mathbb A_{ij}.$$

We summarize the above discussion as the following.
\begin{prop}
For each $\mathcal X\in \{\mathcal X_1, \mathcal X_2\}$, there exists a finite partition  $\mathbb A=\mathbb A_1\cup \mathbb A_2\cdots \cup \mathbb A_n$ where $n=\sum_{i=1} l_i$ such that for each $1\le i\le n$, $\mathbb A_i$ is $ G_0$--invariant and     satisfies projection axioms with projection constant $\xi$.  
\end{prop}

We are now ready to complete the proof of Theorem \ref{thm:main4}.

\begin{proof}[Proof of Theorem \ref{thm:main4}]
By \cite[Induction 2.2]{BBF2}, if a finite index subgroup of $G$ has property (QT) then so does $G$. Thus it suffices to show that $G_0$ has property (QT) where $G_0$ is the finite index subgroup of $\dot G < G$ given by Corollary~\ref{FiniteIndexAdmissibleCor}. By abuse of notations, we denote $G_0$ by $G$, and we remark here that for the rest of the proof, results in Section~\ref{qiefinitequasitrees} and Section~\ref{sec:properaction}  will  apply for $G = G_0$, but not for the original $G$.

By Proposition \ref{prop:qiehyperbolicspaces}, $\dot G$ acts on the product $\mathcal X_1\times \mathcal X_2$ so that the orbital map is quasi-isometrically embedded. Furthermore, there exists a $\dot G$--equivariant quasi-isometric embedding of each $\mathcal X_i$ $(i=1,2)$ into the product of the cone-off space $(\dot{\mathcal X_i}, d_{\dot{\mathcal X}}^K)$ and a quasi-tree by  Proposition \ref{QIembedProp}. Therefore, it suffices to establish  a $G_0 $--equivariant quasi-isometric embedding of  $(\dot{\mathcal X_i}, d_{\dot{\mathcal X}}^K)$ into a finite product of quasi-trees.

By construction, each $\check{\mathbb A}_{i}^j (1\le j\le l_i)$ is $\ddot K_{v_i}$--invariant and $\ddot K_{v_i}$ acts co-boundedly on $\dot Y_v$, so there exists  some $R$ independent of $i, j$ so that the union of quasi-lines in $\check{\mathbb A}_{i}^j$  is $R$--cobounded in $\dot Y_{v_i}$.

Let $x\in \mathcal X$ and $x\in \dot Y_{v_i}$, we choose a point $\Phi_i(x) \in \cup_{\gamma\in \check{\mathbb A}_{i}^j} \gamma$ for $1\le i\le n$ such that $d(x, \Phi_i(x))\le R$.  By $G_0$--equivariance we define $\Phi_i(gx)=g\Phi_i(x)$ for any $g\in G_0$.

By Proposition \ref{ConeoffDistFProp}, the formula (\ref{ConeoffDistFormulaEQ}) holds for any $x, y\in \dot {\mathcal X}$. Note the sum $$\sum_{\gamma\in \mathbb A} [\dot d_\gamma(x, y)]_{K}=\sum_{i=1}^n \sum_{\gamma\in \mathbb A_i} [\dot d_\gamma(x, y)]_{K}$$ 
For each $\mathbb A_i$, let $\mathcal C_K(\mathbb A_i)$ be the quasi-tree of quasi-lines and by Proposition \ref{BBFDistanceThm} 
$$
\sum_{\gamma\in \mathbb A_i} [\dot d_\gamma(x, y)]_{K} \sim d_{\mathcal C_i} (\Phi_i(x), \Phi_i(y))
$$
Hence the formula (\ref{ConeoffDistFormulaEQ}) implies 
$$
(\dot {\mathcal X}, d_{\dot{\mathcal X}}^K)  \to T\times \prod_{i=1}^n \mathcal C_K(\mathbb A_i)  
$$
is a  $G_0$--equivariant quasi-isometric embedding. 
The proof of the Theorem is thus completed. 
\end{proof}

\section{Finite height subgroups in a CKA action $G \curvearrowright X$}
\label{sec:finiteheightquasiconvex}
In this section, we are going to prove Theorem~\ref{thm:generalization} that basically says having finite height and strongly quasiconvexity are equivalent to each other in the context of CKA actions, and both properties can be characterized in term of their group elements. The heart of the proof of this theorem belongs to the implication  $(\ref{thm1:item3}) \implies (\ref{thm1:item1})$ where we use Sisto's notion of path systems (\cite{Sis18}).
We first review some concepts finite height subgroups, strongly quasi-convex subgroups as well as some terminology in \cite{Sis18}.

\begin{defn}
Let $G$ be a group and $H$ a subgroup of $G$. We say that conjugates $g_1Hg_1^{-1}, \cdots g_kHg_k^{-1} $ are \emph{essentially distinct} if the cosets $g_1H,\cdots,g_kH$ are distinct. We call $H$ has {\it height at most $n$} in $G$ if the intersection of any $(n+1)$ essentially distinct conjugates is finite. The least $n$ for which this is satisfied is called the \emph{height} of $H$ in $G$.
\end{defn}

\begin{defn} [Strongly quasiconvex, \cite{Tra19}]
A subset $Y$ of a geodesic space $X$ is called \emph{strongly quasiconvex}  if for every $K \geq 1,C \geq 0$ there is some $M = M(K,C)$ such that every $(K,C)$--quasi–geodesic with endpoints on $Y$ is contained in the $M$--neighborhood of $Y$. 

Let $G$ be a finitely generated group and $H$ a subgroup of $G$.  We say $H$ is \emph{strongly quasiconvex} in $G$ if $H$ is a strongly quasi-convex subset in the Cayley graph $\Gamma(G,S)$ for some (any) finite generating set $S$. A group element $g$ in $G$ is \emph{Morse} if $g$ is of infinite order and the cyclic subgroup generated by $g$ is strongly quasiconvex.
\end{defn}

\begin{rem}
The strong quasiconvexity of a subgroup does not depend on the choice of finite generating sets, and this notion is equivalent to quasiconvexity in the setting of hyperbolic groups. It is shown in \cite{Tra19} (see Theorem~1.2) that strongly quasi-convex subgroups of a finitely generated group are finitely generated and have finite height.
\end{rem}

The following proposition is cited from Proposition~2.3 and Proposition~2.6 in \cite{NTY19}.
\begin{prop}
\label{prop:summarizeNTY19}
\begin{enumerate}
    \item Let $G$ be a group such that the centralizer $Z(G)$ of $G$ is infinite. Let $H$ be a finite height infinite subgroup of $G$. Then $H$ must have finite index in $G$
    
    \item Assume a group $G$ is decomposed as a finite graph $T$ of groups that satisfies the following.
\begin{enumerate}
    \item For each vertex $v$ of $T$ each finite height subgroup of vertex group $G_v$ must be finite or have finite index in $G_v$.
    \item Each edge group is infinite.
\end{enumerate}
Then, if $H$ is a finite height subgroup of G of infinite index, then $gHg^{-1} \cap G_v$ is finite for each vertex group $G_v$ and each group element $g$. In particular, if $H$ is torsion free, then $H$ is a free group.
\end{enumerate}
\end{prop}

\begin{defn}[Path system, \cite{Sis18}]
\label{defn:pathsystem}
Let $X$ be a metric space. A \emph{path system} $\mathcal{PS}(X)$ in $X$ is a collection of $(c,c)$--quasi-geodesic for some $c\ge 1$ such that
 any subpath of a path in $\mathcal{PS}(X)$ is in $\mathcal{PS}(X)$, and all pairs of points in $X$ can be connected by a path in $\mathcal{PS}(X)$.
\end{defn}

\begin{defn}[$\mathcal{PS}$--contracting, \cite{Sis18}]
\label{defn:contracting}
Let $X$ be a metric space and let $\mathcal{PS}(X)$ be a path system in $X$.
A subset $A$ of $X$ is called $\mathcal{PS}(X)$--contracting if there exists $C>0$ and a map $\pi \colon X \to A$ such that
\begin{enumerate}
    \item
    \label{item1:defn:contracting}For any $x \in A$, then $d(x, \pi(x)) \le C$
    \item 
    \label{item2:defn:contracting}For any $x, y \in X$ such that $d(\pi(x), \pi(y)) \ge C$ then for any path $\gamma$ in $\mathcal{PS}(X)$ connecting $x$ to $y$ then $d(\pi(x), \gamma) \le C$ and $d(\pi(y), \gamma) \le C$.
\end{enumerate}
The map $\pi$ will be called $\mathcal{PS}(X)$--projection on $A$ with constant $C$.
\end{defn}

\begin{lem}\cite[Lemma~2.8]{Sis18}
\label{lem:contractingtoquasiconvex}
Let $A$ be a $\mathcal{PS}(X)$--contracting subset of a metric space $X$, then $A$ is strongly quasi-convex.
\end{lem}

\begin{thm}
\label{thm:pathsystems}
 Let $G \curvearrowright X$ be a CKA action. Let $\mathcal{PS}(X)$ be the collection of all special paths defined in Definition~\ref{SpecialPathDefn}. Then $(X, \mathcal{PS}(X))$ is a path system.
\end{thm}
\begin{proof}
The proof   follows from Proposition~\ref{prop:spepathisqg}.
\end{proof}

{\bf For the rest of this section, we fix a CKA action $G \curvearrowright X$ and $G \curvearrowright T$ the action of $G$ on the associated Bass-Serre tree. We also fix the path system  $(X,\mathcal{PS}(X))$ in Theorem~\ref{thm:pathsystems}.}


To get into the proof of Theorem~\ref{thm:generalization}, we need several lemmas.
The following lemma tells us that finite height subgroups in the CKA action $G \curvearrowright X$ are virtually free.

\begin{lem}
\label{lem:virtuallyfree}
 Let $K \le G$ be a nontrivial finitely generated infinite index subgroup of $G$. If $K$ has finite height in $G$, then $K \cap gG_{v}g^{-1}$ is finite for any $v\in T^0$ and $g\in G$. In particular, $K$ is virtually free.
\end{lem}
\begin{proof}
Suppose that $K$ has finite height in $G$. Since the centralizer $Z(G_v)$ each each vertex group is isomorphic to $\Z$, it follows from Proposition~\ref{prop:summarizeNTY19} that for any $g \in G$ and $v \in T^0$, the intersection $K \cap gG_{v}g^{-1}$ is finite. Thus, $K$ acts properly on the tree $T$ and the stabilizer in $K$ of each vertex in $T$  is finite. It follows from \cite[Theorem~7.51]{DK18} that $K$ is virtually free.
\end{proof}

\begin{rem}
\label{rem:constructC_K}
Let $K \le G$ be a nontrivial finitely generated infinite index subgroup of $G$. Suppose that $K$ is a free group of finite rank and every nontrivial element in $K$ is not conjugate into any vertex group. Then there exists a subspace $C_K$ of $X$ such that $K$ acts geometrically on $C_K$ with respect to the induced length metric on $C_K$. The subspace $C_K$ is constructed as the following.

Fix a vertex $v$ in $T$, and fix a point $x_0$ in $Y_v$ such that $\rho(x_0) =v$. Let $\{g_1, g_2, \dots , g_n\}$ be a generating set of $K$.  For each $i \in \{1,2, \dots, n\}$, let $g_{n+i} = g_{i}^{-1}$. Let $\gamma_j$ be the geodesic in $X$ connecting $x_0$ to $g_{j}(x_0)$ with 
$j \in \{1,2, \dots,  2n\}$. Let $C_K$ be the union of segment $g(\gamma_j)$ where $g$ varies elements of $K$ and $j \in \{1, \dots, 2n\}$.
\end{rem}

The following lemma is well-known (see Lemma~2.9 in \cite{CK02} or Lemma~4.5 in \cite{GM14} for proofs).
\begin{lem}
\label{lem:lemma2.9CK02}
Let $X$ be a $\delta$--hyperbolic Hadamard space. Let $\gamma_1$ and $\gamma_2$ be two geodesic lines of $X$ such that $\partial_{\infty} \gamma_{1} \cap \partial_{\infty} \gamma_{2} = \emptyset$. 
 Let $\eta$ be a minimal geodesic segment between $\gamma_1$ to $\gamma_2$. Then any geodesic segment running from $\gamma_1$ to $\gamma_2$ will pass within distance $D = D(\gamma_1, \gamma_2)$ of both endpoints of $\eta$. Moreover, when $d(\gamma_1, \gamma_2) > 4\delta$ then we may take $D =2\delta$.
\end{lem}

Lemma~\ref{lem:projectionbounded} and Lemma~\ref{lem:neighborhood} below are used in the proof of Proposition~\ref{morseimpliesstronglyqc}.

\begin{lem} 
\label{lem:projectionbounded}
{Given a constant $\mu >0$, there exists a constant $r >0$ such that the following holds. Let $e$ and $e'$ be two consecutive edges in $T$ with a common vertex $v$. }
Let $A$ be a subset of $Y_v$ such that  $diam(A) \le \mu$. Suppose that $A \cap F_{e} \neq \emptyset$ and $A \cap F_{e'} \neq \emptyset$.
 
Let $[p, q]$ be the shortest path joining two lines $\ell := \bar{Y}_{v} \cap F_{e}$ to $\ell' : = \bar{Y}_{v} \cap F_{e'}$ where $p\in \ell$ and $q\in \ell'$. For any $x \in F_{e} \cap A$ and $y \in F_{e'} \cap A$, let $u$ and $v$ be the projections of $x$ and $y$ into the lines $\ell$ and $\ell'$ respectively. Then $d(u, p) \le r$ and $d(v, q) \le r$. 
\end{lem}
\begin{proof}
We recall that $Y_v = \bar{Y}_{v} \times \R$ and $H_v$ acts properly and cocompactly on $\bar{Y}_v$. Since $H_v$ is a nonelementary hyperbolic group, it follows that $\bar{Y}_{v}$ is a $\delta_v$--hyperbolic space for some $\delta_v \ge 0$. { Let $\delta$ be the maximum of the hyperbolicity constants of the $\bar{Y}_{v}$'s}.

Let $D = D(\ell, \ell') >0$ be the constant given by Lemma~\ref{lem:lemma2.9CK02}. Let $r(e, e') := 4D + \mu$. Since $\partial_{\infty} \ell \cap \partial_{\infty} \ell' = \emptyset$ and $u \in \ell$, $v \in \ell'$, it follows from Lemma~\ref{lem:lemma2.9CK02}  that there exist $p', q' \in [u, v]$ such that $d(p, p') \le D$ and $d(q, q') \le D$. By the triangle inequality, we have $d(u, p) + d(p, q) + d(q, v) \le 4D + d(u,v)$.
Since $u$ and $v$ are projection points of $x$ and $y$ into the factor $\bar{Y}_{v}$ of $Y_{v} = \bar{Y}_{v} \times \R$ respectively, it follows that $d(u, v) \le d(x,y)$.
Since $x, y \in A$ and $diam(A) \le \mu$, it follows that $d(x, y) \le \mu$. Hence $d(u, v) \le d(x, y) \le \mu$. Thus, $d(u, p) \le 4D + d(u,v)  \le 4D + \mu = r(e, e')$ and $d(v, q) \le 4D + d(u,v) \le 4D + \mu =r(e,e')$.

{By Lemma~\ref{lem:lemma2.9CK02}, we note that whenever the distance between two lines $\bar{Y}_{v} \cap F_{e}$ and $\bar{Y}_{v} \cap F_{e'}$ is at least $4\delta$ then we can define $r(e,e') = 8\delta + \mu$. We remark here that module $G$ there are only finitely many cases $d(F_{e}, F_{e'}) < 4\delta$. Thus  there are only finitely many $r(e,e')$ up to the action of $G$. Let $r$ be the maximum of these constants.}
\end{proof}

\begin{lem}
\label{lem:neighborhood}
 Let $K \le G$ be a finitely generated, finite height subgroup of $G$ of infinite index.   Let $C_K$ be the subspace of $X$ given by Remark~\ref{rem:constructC_K}. Then there exists a constant $R >0$ such that if $\gamma$ is a special path in $X$ (see Definition~\ref{SpecialPathDefn}) connecting two points in $C_K$ then $\gamma \subset {N}_{R}(C_K)$.
\end{lem}

\begin{proof}
By the construction of $C_K$, we note that there exists a constant $\mu >0$ such that $diam(C_K \cap X_{v}) < \mu$ for any vertex $v \in \mathcal{V}(T)$.  {Let $r$ be the constant given by  Lemma~\ref{lem:projectionbounded}.}


Recall that we choose a $G$--equivariant family of Euclidean planes $\{F_e: F_e\subset Y_e\}_{e \in T^1}$. Let $\{\mathcal{S}_{ee'}\}$ be the collection of strips in $X$ given by Section~\ref{section:preliminary}.
For  any three consecutive edges $e, e', e''$ in the tree $T$, two lines $\mathcal{S}_{ee'} \cap F_{e'}$ and $\mathcal{S}_{e'e''} \cap F_{e'}$ in the plane $F_{e'}$ determine an angle in $(0, \pi)$. However, there are only finitely many angles shown up. We denote these angles by $\theta_{1}, \dots, \theta_{k}$. 

Let $D$ be the constant given by Lemma~\ref{defn:vertexedgespace} such that  $X_v= {N}_{D}(Y_{v})$ for every vertex $v \in T^{0}$.
Let 
\[
\xi = 2\mu + r + \max \bigset {\frac{2 \mu + r}{\sin(\theta_j)} + \frac{2\mu + r}{\sin(\pi - \theta_j)}}{j \in \{1, \dots, k\} \, \textup{and} \, \theta_j \neq \pi/2}
\] 
and
\[
R= 2r + \mu + 2\xi + D
\]
Let $x$ and $y$ be the initial and terminal points of $\gamma$. We note that $x \in X_{\rho(x)}$ and $y \in X_{\rho(y)}$.
We consider the following cases:

{\bf Case~1:} $\rho(x) = \rho(y)$.
In this case, the special path $\gamma$ is the geodesic in $X$ connecting $x$ to $y$. 
Since $x, y \in C_K \cap X_{\rho(x)}$ and $diam(C_K \cap X_{\rho(x)}) \le \mu$, it follows that $Len (\gamma ) = d(x, y) \le \mu < R$. 
Thus, $\gamma \subset {N}_{R}(C_K \cap X_{\rho(x)}) \subset {N}_{R}(C_K)$.

{\bf Case~2:} $\rho(x) \neq \rho(y)$. Since $X_u = {N}_{D}(Y_u)$ for any vertex $u\in T^0$, hence without losing of generality, we can assume that $x \in Y_{\rho(x)}$ and $y \in Y_{\rho(y)}$. We recall the construction of the path $\gamma$ from Definition~\ref{SpecialPathDefn}. Let $e_{1} \cdots e_{n}$ be the geodesic edge path connecting  $\rho(x)$ to $\rho(y)$ and let $p_i\in F_{e_i}$ be the intersection point of the strips $\mathcal S_{e_{i-1}e_i}$ and $\mathcal S_{e_{i}e_{i+1}}$, where $e_{0}:=x$ and $e_{n+1}:=y$.  Then
\[
\gamma = [x, p_{1}][p_{1}, p_{2}]\cdots [p_{{n-1}}, p_{n}][p_{n}, y]
\]
Let $p_0 : = x$ and $p_{n+1} : = y$. In order to prove that $\gamma \subset {N}_{R}(C_K)$, we only need to show that $[p_i, p_{i+1}] \subset {N}_{R}(C_K)$ with $i \in \{1, \dots, n\}$.

The proofs for the cases $i =0$ and $i =n$ and for  the cases $i = 1, \dots, n-1$ are similar, so we only need to give the proofs for the cases $i =0, 1$.

{\underline {Proof of case $i =0$:}}

Let $v_i$ be the initial vertex of $e_i$ (with $i \in \{1, \dots, n \}$), and $v_n$ be the terminal vertex of $e_n$. We recall that two lines $\mathcal{S}_{xe_1} \cap F_{e_1}$ and $F_{e_1} \cap \bar{Y}_{v_0}$ in the plane $F_{e_1}$ are perpendicular.
Since $C_K \cap F_{e_1} \neq \emptyset$, we choose a point $O_1 \in Y \cap F_{e_1}$.

{\bf Claim:} $d(O_1, p_1) < r + \xi$.
\begin{proof}[Proof of the claim]
Let $\bar{O}_1$ be the projection of $O_1$ into the line $F_{e_1} \cap \bar{Y}_{v_0}$. Let $\bar{V}_1$ be the projection of $O_1$ into the line $F_{e_1} \cap \bar{Y}_{v_1}$. By Lemma~\ref{lem:projectionbounded}, we have
\begin{equation}
\label{eq:1}
d(\bar V_{1}, \mathcal{S}_{e_1e_2} \cap F_{e_1} \cap \bar{Y}_{v_1}) \le r
\end{equation} (we note that $ \mathcal{S}_{e_1e_2} \cap F_{e_1} \cap \bar{Y}_{v_1} = \gamma_{e_1e_2}(0)$).
Since $O_1$ and $p_0 =x$ belong to $X_{v_0} \cap C_K$ and $diam(X_{v_0} \cap C_K) \le \mu$, it follows that $d(O_1, p_0) \le \mu$.
Let $\bar{p}_0$ be the projection of $p_0$ into the factor $\bar{Y}_{v_0}$ of $Y_{v_0}  = \bar{Y}_{v_0} \times \R$. We have that $d(\bar O_1, \bar p_0) \le d(O_1, p_0) \le \mu$. Since $d(\bar p_0, \mathcal{S}_{xe_1} \cap F_{e_1} \cap \bar{Y}_{v_0})$ is the minimal distance from $\bar p_0$ to the line $F_{e_1} \cap \bar{Y}_{v_0}$ and $\bar{O}_{1} \in F_{e_1} \cap \bar{Y}_{v_0}$ we have that $d(\bar p_0, \mathcal{S}_{xe_1} \cap F_{e_1} \cap \bar{Y}_{v_0}) \le d(\bar p_0, \bar O_1) \le \mu$. Using the triangle inequality for three points $\bar p_0$, $\bar O_1$, and $\mathcal{S}_{xe_1} \cap F_{e_1} \cap \bar{Y}_{v_0}$, we have
\begin{equation}
    \label{eq:2}
 d(\bar O_1,   \mathcal{S}_{xe_1} \cap F_{e_1} \cap \bar{Y}_{v_0}) \le 2\mu
\end{equation}

Let $A$ be the projection of $O_1$ into the line $F_{e_1} \cap \mathcal{S}_{e_1e_2}$. Using formula~(\ref{eq:1}), we have
\begin{equation}
    \label{eq:3}
    d(O_1, A)  = d(V_1, \mathcal{S}_{e_1e_2} \cap F_{e_1} \cap \bar{Y}_{v_1}) \le r
\end{equation}

Let $T$ be the projection of $O_1$ into the line $\mathcal{S}_{xe_1} \cap F_{e_1}$. Using formula~(\ref{eq:2}), we have
$d(O_1, T) = d(\bar O_1, \mathcal{S}_{xe_1} \cap F_{e_1} \cap \bar{Y}_{v_0}) \le 2\mu$. Thus, we have $d(A,T) \le d(A, O_1) + d(O_1, T) \le r + 2\mu$.
An easy application of Rule of Sines to the triangle $\Delta(T, p_1, A)$ together with the fact $d(A, T) \le 2\mu +r$ give us that $d(p_1, A) < \xi$ and $d(p_1, T) < \xi$. 
Combining these inequalities with formula~(\ref{eq:3}), we obtain that
$d(O_1, p_1) \le d(O_1, A) + d(A, p_1) < r + \xi$ The claim is verified.
\end{proof}
Using the facts $d(O_1, p_0) \le \mu$ and $d(O_1, p_1) < r + \xi$ we have
$d(p_0, p_1) \le d(p_0, O_1) + d(O_1, p_1) \le \mu + r + \xi < R$.
Since $p_0 = x \in C_K$, it follows that $[p_0, p_1] \subset {N}_{R}(C_K)$.

{\underline{Proof of the case $i =1$:}}

Since $C_K \cap F_{e_2} \neq \emptyset$, we choose a point $O_2 \in C_K \cap F_{e_2}$. Since $O_1, O_2$ belong to $C_{K} \cap Y_{v_1}$ and $diam(C_K \cap Y_{v_1}) \le \mu$, we have $d(O_1, O_2) \le \mu$.
By a similar argument as in the proof of the claim of  the case $i =0$, we can show that $d(O_2, p_2) < r + \xi$. Thus,
$ d(p_1, p_2) \le d(p_1, O_1) +d(O_1, O_2) + d(O_2, p_2)
    < (r +\xi) + \mu + (r+ \xi) 
    = 2r + \mu + 2\xi $
Since $O_1, O_2 \in C_K$, it is easy to see that $[p_1, p_2] \subset {N}_{3r + \mu+ 3\xi}(Y) \subset  {N}_{R}(Y)$.
\end{proof}

We recall that an infinite order element $g$ in a finitely generated group is Morse if the cyclic subgroup generated by $g$ is strongly quasi-convex. 
\begin{lem}
\label{lem:Moresenotconjugate}
If an infinite order element $g$ in $G$ is More, then it is not conjugate into any vertex group of $G$.
\end{lem}
\begin{proof}
Since $g$ is Morse, it follows that the infinite cyclic subgroup $\langle g \rangle$ generated by $g$ is strongly quasi-convex in $G$. We would like to show that $g$ is not conjugate into any vertex group. Indeed, by way of contradiction, we assume that $g \in xG_{v} x^{-1}$ for some $x \in G$ and for some vertex group $G_v$. Hence, the cyclic subgroup generated by $h = x^{-1} g x$ is strongly quasi-convex in $G$. Since $G_v$ is undistorted in $G$ (as $G_v$ acts geometrically on $Y_v$ and $Y_v$ is undistorted in $X$), it follows from Proposition~4.11 in \cite{Tra19} that  $\langle h \rangle $ is strongly quasi-convex in $G_v$. By Theorem~1.2 in \cite{Tra19},  $\langle h \rangle $ has finite height in $G_v$. Since the centralizer $Z(G_v)$ of $G_v$ is isomorphic to $\Z$, it follows from Proposition~\ref{prop:summarizeNTY19} that  $\langle h \rangle $ has finite index in $G_v$. This contradicts to the fact that $G_v$ is not virtually cyclic group. Therefore $g$ is not conjugate into any vertex group of $G$.
\end{proof}


\begin{prop}
\label{morseimpliesstronglyqc}
 Let $K$ be a finitely generated free subgroup of $G$ of infinite index such that all nontrivial elements in $K$ are Morse in $G$. Choose a vertex $v$ in a minimal $K$--invariant subtree $T'$ of $T$. Let $C_K$ be the subspace of $X$ given by  Remark~\ref{rem:constructC_K} with respect to a generating set $\{g_1, g_2, \dots, g_n\}$ of $K$. Then $C_K$ is contracting in $(X, \mathcal{PS}(X))$. As a consequence, $K$ is strongly quasi-convex in $G$.
\end{prop}
\begin{proof}
{We first recall the construction of $C_K$ from Remark~\ref{rem:constructC_K}. We first fix a point $x_0$ in $Y_v$ such that $\rho(x_0) =v$. For each $i \in \{1,2, \dots, n\}$, let $g_{n+i} = g_{i}^{-1}$. Let $\gamma_j$ be the geodesic in $X$ connecting $x_0$ to $g_{j}(x_0)$ with 
$j \in \{1,2, \dots,  2n\}$. Then $C_K$ is the union of segment $g(\gamma_j)$ where $g$ varies elements of $K$ and $j \in \{1, \dots, 2n\}$.  }
Since $K$ is a free subgroup of $G$ and all nontrivial elements in $K$ are Morse in $G$, it follows from Lemma~\ref{lem:Moresenotconjugate} that every nontrivial element in $K$ is not conjugate into any vertex group $G_v$. 
Hence,  $K$ acts freely on the Bass-Serre tree $T$.
To show that $C_K$ (we note that $K(x_0) \subset C_K$) is contracting in $(X, \mathcal{PS}(X))$, we need to define a $\mathcal{PS}(X)$--projection  $\pi \colon X \to C_K$ satisfying conditions (\ref{item1:defn:contracting}) and (\ref{item2:defn:contracting}) in Definition~\ref{defn:contracting}.



{\bf Step~1: Constructing $\mathcal{PS}(X)$--projection $\pi \colon X \to C_K$ on $C_K$.}

Let  $\rho \colon X \to T^0$ be the indexed $G$-map given by Remark~\ref{rem:indexfunction}, which is coarsely $L$-lipschitez.  Let $\mathcal{R} \colon T \to T'$ be the nearest point projection from $T$ to the minimal $K$-invariant subtree $T'$.  Since $K$ acts on the minimal tree $T'$ cocompactly, it follows that there exists a constant $\delta' >0$ such that $T' \subset N_{\delta'} (K(v))$. 

Let $x$ be any point in $X$. Choose an element $g \in K$ such that $\mathcal{R}(\rho(x))\in B(g(v), \delta')$. We define $\pi(x) : = g(x_0)$.

{\bf Step~2:  Verifying the condition~(\ref{item1:defn:contracting}) in Definition~\ref{defn:contracting}.}
{Recall that $\gamma_i$ is the geodesic in $X$ connecting $x_0$ to $g_{i}(x_0)$}. 
Let $\delta = \max \bigset{Len( \gamma_i )}{ i \in \{1, \dots, n \}}$.    Let $\mu >0$ be a constant such that $diam(C_K \cap Y_u) \le \mu$ for any vertex $u \in T^0$. Let $R$ be the constant given by Lemma~\ref{lem:neighborhood}.
Since $K$ acts cocompactly both on $C_K$ and $T'$, there exists a constant $\epsilon \ge 1$ such that for any $h$ and $h'$ in $K$ then 
\begin{equation}
\label{eq:quasiC_K}
    d \bigl (h(x_0), h'(x_0) \bigr )\, \bigl /\epsilon - \epsilon \le d_{T}\bigl (h(v), h'(v) \bigr ) \le \epsilon d \bigl (h(x_0), h'(x_0) \bigr ) + \epsilon
\end{equation}

{\bf Claim~1:} Let $C$ be a sufficiently large constant such that $  \delta+ \epsilon ( \epsilon + 2(L \delta + L) + \delta') + (5+2\delta') \mu + R  < C$. Then $d(x, \pi(x)) \le C$ for any $x \in C_K$.

Indeed, since $C_{K} \subset {N}_{\delta}(K(x_0))$, there exists $k \in K$ such that $d(x, k(x_0)) \le \delta$. 
Recalling that $\rho$ is a $K$-equivariant, coarsely $L$-lipschitez map,  and $v=\rho(x_0)$, we have
$$d_{T}(\rho(x), k(v)) = d(\rho(x), \rho(k(x_0))
    \le L d(x, k(x_0)) + L \le L\delta + L.$$ 
The nearest point projection of $\mathcal{R} \colon T \to T'$ implies
$
d_{T}(\rho(x), \mathcal{R}(\rho(x))) \le d_{T}(\rho(x), k(v)) \le L\delta + L
$. Hence, {$d_{T}(\rho(x), g(v)) \le d_{T}(\rho(x), \mathcal{R}(\rho(x))) +d(\mathcal{R}(\rho(x)), g(v)) \le (L\delta + L) + \delta'$}. It implies that {$d_{T}(k(v), g(v)) \le d_{T}(k(v), \rho(x)) + d_{T}(\rho(x), g(v)) \le  2(L\delta + L) + \delta'$.}
Putting the above inequalities together with formula~(\ref{eq:quasiC_K}), we have
\begin{align*}
    d(x, \pi(x)) = d(x, g(x_0))  &\le d(x, k(x_0)) + d(k(x_0), g(x_0))
    \le \delta + \epsilon( \epsilon + d(k(v), g(v))) \\
    &\le \delta+ \epsilon ( \epsilon + 2(L \delta + L) + \delta') < C
\end{align*}
Claim~1 is confirmed.

{\bf Step~3: Verifying condition~(\ref{item2:defn:contracting}) in Definition~\ref{defn:contracting}.}

{\bf Claim~2:} Let $C$ be the constant given by Claim~1. Then the projection $\pi \colon X \to C_K$ satisfies  condition~(\ref{item2:defn:contracting}) in Definition~\ref{defn:contracting} with respect to this constant $C$.

Let $x$ and $y$ be two points in $X$ such that $d(\pi(x), \pi(y))  \ge C$. Let $\gamma$ be a special path in $X$ connecting $x$ to $y$. We would like to show that $d(\pi(x), \gamma) \le C$ and $d(\pi(y), \gamma) \le C$.

We recall that $X_{u} = {N}_{D}(Y_u)$ for any vertex $u \in T^0$. Thus we assume, without loss of generality that $x \in Y_{\rho(x)}$ and $y \in Y_{\rho(y)}$.  Recall that $\pi(x) = g(x_0)$ and $\pi(y) = g'(x_0)$ where $g , g' \in K$  such that $\mathcal{R}(\rho(x))$ and $\mathcal{R}(\rho(y))$ are in the balls $B(g(v), \delta')$ and $B(g'(v), \delta')$ respectively.  
We have  
$\epsilon^2 + (20 + 4\delta') \epsilon < C \le d(\pi(x), \pi(y)) = d(g(x_0), g'(x_0)) \le \epsilon (\epsilon + d_{T}(g(v), g'(v))$. Hence $20 + 4\delta' < d_{T}(g(v), g'(v))$. Since $\mathcal{R}(\rho(x)) \in B(g(v), \delta')$ and $\mathcal{R}(\rho(y)) \in B(g'(v), \delta')$, we have $d_{T}(\mathcal{R}(\rho(x)), \mathcal{R}(\rho(y)) ) > 20 + 2\delta'$. 

Choose vertices $\sigma$, $\tau$ in the geodesic $[\mathcal{R}(\rho(x)), \mathcal{R}(\rho(y))]$ such that $d_{T}(\sigma, \mathcal{R}(\rho(x)))$ and $d_{T}(\tau, \mathcal{R}(\rho(y)))$ are the smallest integers bigger than or equal to $3 + \delta'$. Thus $d_{T}(g(v), \sigma) \ge d_{T}(\mathcal{R}(\rho(x)), \sigma) - d_{T}(\mathcal{R}(\rho(x)), g(v)) \ge d_{T}(\mathcal{R}(\rho(x)), \sigma) - \delta' \ge 3$. Similarly, we have $d_{T}(g'(v), \tau) \ge 3$.

Let $\alpha$ be a special path in $X$ connecting $g(x_0) \in X_{g(v)}$ to $g'(x_0) \in X_{g'(v)}$. We remark here that there exist a subpath $\gamma'$ of the special path $\gamma$  and a subpath $\alpha'$ of $\alpha$ such that $\gamma'$ and $\alpha'$ connect some point in $X_{\sigma}$ to some point in $X_{\tau}$. By Remark~\ref{rem:independent}, we have that $\alpha' \cap X_{u} = \gamma' \cap X_{u}$ for any vertex $u$ in the geodesic $[\sigma, \tau]$.



Let $R$ be the constant given by Lemma~\ref{lem:neighborhood}. It follows from Lemma~\ref{lem:neighborhood}( see Case~1 and Case~2 in this lemma) that $\alpha \cap Y_{\sigma} \subset {N}_{R}(C_K \cap Y_{\sigma})$.
Choose $z \in \alpha \cap Y_{\sigma}$ and $z' \in C_K \cap X_{\sigma}$ such that $d(z, z') \le R$. Since $d_{T}(g(v), \sigma) \le d_{T}(g(v), \mathcal{R}(\rho(x))) + d_{T}(\mathcal{R}(\rho(x)), \sigma) \le \delta' + (4 + \delta') = 2\delta' +4$, and $g(x_0) \in C_{K} \cap X_{g(v)}$ and $z' \in C_{K} \cap X_{\sigma}$, it follows that $d(g(x_0), z') \le (5 + 2\delta') \mu$. Hence,
$$
d(\pi(x), \gamma) \le d(\pi(x), z)
    \le d(\pi(x), z') + R = d(g(x_0), z') + R \le (5+2\delta') \mu + R < C
$$
Similarly, we can show that $d(\pi(y), \gamma) < C$. Thus  Claim~2 is established. 

Combining Step~1, Step~2, and Step~3 together, we conclude that $C_K$ is contracting in $(X, \mathcal{PS}(X)$. By Lemma~\ref{lem:contractingtoquasiconvex}, $C_K$ is strongly quasi-convex in $X$. We conclude that $K$ is strongly quasi-convex in $G$. The proposition is proved.
\end{proof}

Proposition~\ref{morseimpliesstronglyqc} has the following corollary that characterize Morse elements and contracting element in the admissible group $G$.
\begin{cor}
\label{cor:Morse}
\begin{enumerate}
    \item \label{cor:Morse:item1} An infinite order element $g$ in $G$ is Morse if and only  if $g$ is not conjugate into any vertex group $G_v$.
    \item \label{cor:Morse:item2} An element of $G$ is contracting with respect to $(X, \mathcal{PS}(X)$ if and only if its acts hyperbolically on the Bass-Serre tree $T$.
\end{enumerate}
\end{cor}
\begin{proof}
(\ref{cor:Morse:item1}). By Lemma~\ref{lem:Moresenotconjugate}, if $g$ is Morse in $G$ then it is not conjugate into vertex group of $G$.
Conversely, if $g$ is not conjugate into any vertex group of $G$ then $g$ acts hyperbolically on the Bass-Serre tree $T$. If $g \in G$ acts hyperbolically on the Bass-Serre tree $T$. Let $K$ be the infinite cyclic subgroup generated by $g$. Since $g \in G$ acts hyperbolically on the Bass-Serre tree $T$, it is not conjugate into any vertex subgroup. Fix a point $x_0$ in $X$, by Proposition~\ref{morseimpliesstronglyqc} the orbit space $K(x_0)$ is contracting in $(X, \mathcal{PS}(X))$ (because $C_K$ is contracting in $(X, \mathcal{PS}(X))$). Thus $g$ is a contracting element with respect to $(X, \mathcal{PS}(X))$.  By Lemma~2.8 in \cite{Sis18}, $g$ must be Morse.

(\ref{cor:Morse:item2}). If $g \in G$ acts hyperbolically on the Bass-Serre tree $T$. Let $K$ be the infinite cyclic subgroup generated by $g$. Since $g \in G$ acts hyperbolically on the Bass-Serre tree $T$, it is not conjugate into any vertex subgroup. Fix a point $x_0$ in $X$, by Proposition~\ref{morseimpliesstronglyqc} the orbit space $K \cdot x_0$ is contracting in $(X, \mathcal{PS}(X))$ (because $C_K$ is contracting in $(X, \mathcal{PS}(X))$). Thus $g$ is a contracting element with respect to $(X, \mathcal{PS}(X))$.

Conversely, if $g \in G$ is contracting with respect to $(X, \mathcal{PS}(X))$, then $g$ is Morse. By the assertion~(\ref{cor:Morse:item1}), $g$ is not conjugate into any vertex group. Thus it acts hyperbolically on the Bass-Serre tree $T$.
\end{proof}



We now ready to prove Theorem~\ref{thm:generalization}. 

\begin{proof}[Proof of Theorem~\ref{thm:generalization}]
We are going to prove the following implications: $(\ref{thm1:item1}) \implies (\ref{thm1:item2})$, $(\ref{thm1:item2}) \implies (\ref{thm1:item3})$, $(\ref{thm1:item3}) \implies (\ref{thm1:item1})$, and $(\ref{thm1:item3}) \Longleftrightarrow (\ref{thm1:item4})$.

 $(\ref{thm1:item1}) \implies (\ref{thm1:item2})$:
The implication just follows from 
Theorem~1.2 in \cite{Tra19}. 

$(\ref{thm1:item2}) \implies (\ref{thm1:item3})$: By Lemma~\ref{lem:virtuallyfree}, $K$ has a finite index subgroup $K'$ such that $K'$ is a free group of finite rank.
Let $x$ be an infinite order element in $K'$. By way of contradiction, suppose that $x$ is not Morse in $G$. By Corollary~\ref{cor:Morse},  $x$ is conjugate into a vertex group of $G$. In other words, $x \in gG_{v}g^{-1}$ for some vertex group $G_v$ and for some $g \in G$. Hence $x \in K' \cap gG_{v}g^{-1} \subset K \cap gG_{v}g^{-1}$ that is finite by Proposition~\ref{prop:summarizeNTY19}. So, $x$ has finite order that contradicts to our assumption that $x$ is an infinite order element. Since any infinite order element in $K$ has a power that belongs to $K'$, the implication $(\ref{thm1:item2}) \implies (\ref{thm1:item3})$ is verified.

 $(\ref{thm1:item3}) \implies (\ref{thm1:item1})$:   Let $K'$ be a finite index subgroup of $K$ such that $K'$ is free.
 It follows from Proposition~\ref{morseimpliesstronglyqc} that $K'$ is strongly quasi-convex in $G$. Since $K'$ is a finite index subgroup of $K$, it follows that $K$ is also strongly quasi-convex in $G$.

$(\ref{thm1:item3}) \implies (\ref{thm1:item4})$: Let $K'$ be a finite index subgroup of $K$ such that $K'$ is free. Let $T' \subset T$ be the minimal subtree of $T$ that contains $K'(v)$. Since the map $K' \to T'$ given by $k \to k(v)$ is quasi-isometric and the inclusion map $T' \to T$ is also a quasi-isometric embedding, it follows that the composition $K' \to T' \to T$ is a quasi-isometric embedding. Since $K'$ is a finite index subgroup of $K$, it follows that the map $K \to T$ given by $k \to k(v)$ is a quasi-isometric embedding.

$(\ref{thm1:item4}) \implies (\ref{thm1:item3})$:  By way of contradiction, suppose that there exists an infinite order element $g \in K$ such that $g$ is not Morse in $G$. It follows from Corollary~\ref{cor:Morse} that $g$ is conjugate into a vertex group, hence $g$ fixes a vertex $v$ of $T$. By our assumption, there is a vertex $w$ in $T$ such that the map $K \to T$ given by $k \to k(w)$ is a quasi-isometric embedding. It implies that the map $K \to T$ given by $k \to k(v)$ is also a $(\lambda, \lambda)$--quasi-isometric embedding for some $\lambda > 0$. Choose an integer $n$ large enough such that $|g^n| > \lambda^2$. We have $1/\lambda |g^n| - \lambda \le d(g(v), g^n(v)) = d(v, v) =0$. Hence $|g^n| < \lambda^2$. This contradicts to our choice of $n$.
\end{proof}


\bibliographystyle{alpha}
\bibliography{CKadmissible}

\begin{thebibliography}{GMRS98}

\bibitem[Baj07]{B07}
Jitendra Bajpai.
\newblock Omnipotence of surface groups.
\newblock Master's thesis, McGill University, 2007.

\bibitem[BBF15]{BBF}
Mladen Bestvina, Ken Bromberg, and Koji Fujiwara.
\newblock Constructing group actions on quasi-trees and applications to mapping
  class groups.
\newblock {\em Publ. Math. Inst. Hautes \'{E}tudes Sci.}, 122:1--64, 2015.

\bibitem[BBF19]{BBF2}
Mladen {Bestvina}, Kenneth {Bromberg}, and Koji {Fujiwara}.
\newblock {Proper actions on finite products of quasi-trees}.
\newblock {\em arXiv e-prints}, page arXiv:1905.10813, May 2019.

\bibitem[Bes]{Bes}
Mladen Bestvina.
\newblock Questions in geometric group theory.
\newblock M. Bestvina’s home page, 2004.

\bibitem[BH99]{BH99}
Martin~R. Bridson and Andr\'{e} Haefliger.
\newblock {\em Metric spaces of non-positive curvature}, volume 319 of {\em
  Grundlehren der Mathematischen Wissenschaften [Fundamental Principles of
  Mathematical Sciences]}.
\newblock Springer-Verlag, Berlin, 1999.

\bibitem[BN08]{BN08}
Jason~A. Behrstock and Walter~D. Neumann.
\newblock Quasi-isometric classification of graph manifold groups.
\newblock {\em Duke Math. J.}, 141(2):217--240, 2008.

\bibitem[Bow08]{Bow08}
Brian~H. Bowditch.
\newblock Tight geodesics in the curve complex.
\newblock {\em Invent. Math.}, 171(2):281--300, 2008.

\bibitem[Bow12]{Bow12}
B.~H. Bowditch.
\newblock Relatively hyperbolic groups.
\newblock {\em Internat. J. Algebra Comput.}, 22(3):1250016, 66, 2012.

\bibitem[CK00]{CK00}
Christopher~B. Croke and Bruce Kleiner.
\newblock Spaces with nonpositive curvature and their ideal boundaries.
\newblock {\em Topology}, 39(3):549--556, 2000.

\bibitem[CK02]{CK02}
C.~B. Croke and B.~Kleiner.
\newblock The geodesic flow of a nonpositively curved graph manifold.
\newblock {\em Geom. Funct. Anal.}, 12(3):479--545, 2002.

\bibitem[DGO17]{DGO}
F.~Dahmani, V.~Guirardel, and D.~Osin.
\newblock Hyperbolically embedded subgroups and rotating families in groups
  acting on hyperbolic spaces.
\newblock {\em Mem. Amer. Math. Soc.}, 245(1156):v+152, 2017.

\bibitem[DJ99]{DJ99}
A.~Dranishnikov and T.~Januszkiewicz.
\newblock Every {C}oxeter group acts amenably on a compact space.
\newblock In {\em Proceedings of the 1999 {T}opology and {D}ynamics
  {C}onference ({S}alt {L}ake {C}ity, {UT})}, volume~24, pages 135--141, 1999.

\bibitem[DK18]{DK18}
Cornelia Dru\c{t}u and Michael Kapovich.
\newblock {\em Geometric group theory}, volume~63 of {\em American Mathematical
  Society Colloquium Publications}.
\newblock American Mathematical Society, Providence, RI, 2018.
\newblock With an appendix by Bogdan Nica.

\bibitem[DS05]{DS05}
Cornelia Dru\c{t}u and Mark Sapir.
\newblock Tree-graded spaces and asymptotic cones of groups.
\newblock {\em Topology}, 44(5):959--1058, 2005.
\newblock With an appendix by Denis Osin and Mark Sapir.

\bibitem[Far98]{Farb}
Benson Farb.
\newblock Relatively hyperbolic groups.
\newblock {\em Geometric and Functional Analysis}, 8:810--840, 1998.

\bibitem[GM14]{GM14}
Craig~R. Guilbault and Christopher~P. Mooney.
\newblock Boundaries of {C}roke--{K}leiner-admissible groups and equivariant
  cell-like equivalence.
\newblock {\em J. Topol.}, 7(3):849--868, 2014.

\bibitem[GMRS98]{GMRS}
Rita Gitik, Mahan Mitra, Eliyahu Rips, and Michah Sageev.
\newblock Widths of subgroups.
\newblock {\em Trans. Amer. Math. Soc.}, 350(1):321--329, 1998.

\bibitem[GP16]{GP16}
Victor Gerasimov and Leonid Potyagailo.
\newblock Quasiconvexity in relatively hyperbolic groups.
\newblock {\em J. Reine Angew. Math.}, 710:95--135, 2016.

\bibitem[Gre16]{Gre16}
Sebastian Grensing.
\newblock Virtual boundaries of {H}adamard spaces with admissible actions of
  higher rank.
\newblock {\em Math. Z.}, 284(1-2):1--22, 2016.

\bibitem[HP15]{HP15}
Mark~F. Hagen and Piotr Przytycki.
\newblock Cocompactly cubulated graph manifolds.
\newblock {\em Israel J. Math.}, 207(1):377--394, 2015.

\bibitem[HS13]{HS13}
David Hume and Alessandro Sisto.
\newblock Embedding universal covers of graph manifolds in products of trees.
\newblock {\em Proc. Amer. Math. Soc.}, 141(10):3337--3340, 2013.

\bibitem[HW08]{HW08}
Fr\'{e}d\'{e}ric Haglund and Daniel~T. Wise.
\newblock Special cube complexes.
\newblock {\em Geom. Funct. Anal.}, 17(5):1551--1620, 2008.

\bibitem[KL98]{KL98}
M.~Kapovich and B.~Leeb.
\newblock {$3$}-manifold groups and nonpositive curvature.
\newblock {\em Geom. Funct. Anal.}, 8(5):841--852, 1998.

\bibitem[Liu13]{Liu13}
Yi~Liu.
\newblock Virtual cubulation of nonpositively curved graph manifolds.
\newblock {\em J. Topol.}, 6(4):793--822, 2013.

\bibitem[NTY]{NTY19}
Hoang~Thanh {Nguyen}, Hung~Cong {Tran}, and Wenyuan {Yang}.
\newblock {Quasiconvexity in 3-manifold groups}.
\newblock arXiv:1911.07807. To appear in Mathematische Annalen.

\bibitem[Osi16]{Osin}
D.~Osin.
\newblock Acylindrically hyperbolic groups.
\newblock {\em Trans. Amer. Math. Soc.}, 368(2):851--888, 2016.

\bibitem[PW14]{PW14}
Piotr Przytycki and Daniel~T. Wise.
\newblock Graph manifolds with boundary are virtually special.
\newblock {\em J. Topol.}, 7(2):419--435, 2014.

\bibitem[Sis13]{S13}
Alessandro Sisto.
\newblock Projections and relative hyperbolicity.
\newblock {\em Enseign. Math. (2)}, 59(1-2):165--181, 2013.

\bibitem[Sis18]{Sis18}
Alessandro Sisto.
\newblock Contracting elements and random walks.
\newblock {\em J. Reine Angew. Math.}, 742:79--114, 2018.

\bibitem[Tra19]{Tra19}
Hung~Cong Tran.
\newblock On strongly quasiconvex subgroups.
\newblock {\em Geom. Topol.}, 23(3):1173--1235, 2019.

\bibitem[Wis00]{Wise00}
Daniel~T. Wise.
\newblock Subgroup separability of graphs of free groups with cyclic edge
  groups.
\newblock {\em Q. J. Math.}, 51(1):107--129, 2000.

\bibitem[Wis20]{Wise20}
Daniel~T. Wise.
\newblock {\em The Structure of Groups with a Quasiconvex Hierarchy}, volume
  AMS-209.
\newblock Annals of Mathematics Studies, 2020.

\bibitem[Yan19]{YANG}
Wen-yuan Yang.
\newblock Statistically convex-cocompact actions of groups with contracting
  elements.
\newblock {\em Int. Math. Res. Not. IMRN}, (23):7259--7323, 2019.

\end{thebibliography}
\end{document}